\documentclass[12pt]{amsart}
\numberwithin{equation}{section}
\usepackage{orcidlink}
\usepackage{graphicx,color}
\usepackage{xcolor}
\usepackage{amssymb}
\def\cb{{\mathcal B}}

\def\cd{{\mathcal D}}

\def\cf{{\mathcal F}}

\def\ch{{\mathcal H}}

\def\co{{\mathcal O}}
\def\cp{{\mathcal P}}

\def\cs{{\mathcal S}}

\def\cu{{\mathcal U}}

\def\ga{{\mathfrak A}} 
 
\def\gc{{\mathfrak C}}

\def\gam{{\mathfrak M}}
\def\gpn{{\mathfrak n}}

\def\bc{{\mathbb C}}

\def\bn{{\mathbb N}}

\def\br{{\mathbb R}}
\def\bt{{\mathbb T}}

\def\bz{{\mathbb Z}}

\def\a{\alpha}
\def\b{\beta}
 
\def\d{\delta}  \def\D{\Delta}

\def\l{\lambda} 
\def\k{\kappa}
\def\m{\mu}

\def\n{\nu}
\def\r{\rho}
\def\s{\sigma} \def\S{\Sigma}

\def\f{\varphi} 
\def\th{\theta}  
\def\om{\omega} 
\def\z{\zeta}

\def\id{\hbox{id}}

\newtheorem{thm}{Theorem}[section]
\newtheorem{lem}[thm]{Lemma}

\newtheorem{prop}[thm]{Proposition}
\newtheorem{defin}[thm]{Definition}
\newtheorem{rem}[thm]{Remark}

\newtheorem{exa}[thm]{Example}

\def\aut{\mathop{\rm Aut}}

\newcommand{\ty}[1]{\mathop{\rm {#1}}}
\def\di{\mathop{\rm d}\!}

\def\idd{{1}\!\!{\rm I}}

\def\tr{\mathop{\rm Tr}}

\begin{document}

\title[quasi-invariant states]
{Dominance and equivalence for states on $C^*$-algebras: Quasi-Invariant states}

\author[A. Dhahri, F. Fidaleo, C. K. Ko, H. J. Yoo]{Ameur Dhahri${}^{\orcidlink{0000-0003-0580-7025}}$}
\address{Ameur Dhahri\\
Mathematics Department \\
Politecnico di Milano\\
Piazza Leonardo da Vinci 32, I - 20133 Milano, Italy} \email{{\tt
ameur.dhahri@polimi.it}}
\author[]{Francesco Fidaleo${}^{\orcidlink{0000-0001-6807-2826}}$}
\address{Francesco Fidaleo\\
Dipartimento di Matematica \\
Universit\`{a} di Roma Tor Vergata\\
Via della Ricerca Scientifica 1, Roma 00133, Italy} \email{{\tt
fidaleo@mat.uniroma2.it}}
\author[]{Chul Ki Ko${}^{\orcidlink{0000-0002-7104-5279}}$}
\address{Chul Ki Ko\\
University College\\
Yonsei University\\
85 Songdogwahak-ro, Yeonsu-gu, Incheon 21983, Korea} \email{{\tt
kochulki@yonsei.ac.kr}}
\author[]{Hyun Jae Yoo${}^{\orcidlink{ 0000-0003-0683-2898}}$}
\address{Hyun Jae Yoo\\
Department of Applied Mathematics and Institute for Integrated Mathematical Sciences \\
Hankyong National University\\
327 Jungang-ro, Anseong-si, Gyeonggi-do 17579, Korea} \email{{\tt
yoohj@hknu.ac.kr}}

\begin{abstract}
In view of natural applications of the notion of quasi-invariance for states under the action of a group on $C^*$-algebras, we study a suitable noncommutative generalisation of dominance and equivalence of measure-theoretic nature. For a state dominated (more precisely, {\it strongly absolutely continuous}, using the terminology in \cite{G})
by a fixed one, we derive the corresponding ``Radon-Nikodym'' derivative which is a, in general, unbounded element associated (i.e., affiliated, in the operator algebra language) with the commutant algebra generated by the Gelfand-Naimark-Segal (GNS for short) representation of the dominating one.
Among the main results, we show that such a notion of dominance is, in general, not transitive. In fact, this unpleasant ``pathology'' may well naturally appear since it is associated to the fact that the product of closed operators (i.e. the corresponding noncommutative version of the Radon Nikodym Derivative) might be non-closable.
The investigation of quasi-invariance under such a notion of dominance of a state under the action of a group by $*$-automorphisms is then perfectly meaningful. We show that non-transitivity does not appear: the orbit under the action of a group of a fixed quasi-invariant state is made of states which are all mutually equivalent. Conversely, the closure of an orbit of a quasi-invariant state might contain ``singular'' ones. This means that, in general,
the set of quasi-invariant states cannot be topologically closed, whereas it is closed under convex combinations.
We also provide a unitary implementation (which is, in general, not a representation) of the action of the group on the Hilbert space of the GNS representation of the quasi-invariant state, generalising the GNS covariant representation associated to an invariant one. In view of applications to Quantum Statistical Mechanics, we also compare our framework of dominance with the Pedersen-Takesaki construction in which the associated Radon-Nikodym Derivative exhibits nice cocycle properties, but is situated in the von Neumann algebra (or, more precisely, in the centraliser) generated by the GNS representation of a state with central support in the bidual, instead of in the commutant. 

\end{abstract}

\maketitle

\section{introduction}

The notion of ``dominance'' between measures living on a common Borel Space $X$ is one of the most studied questions in Functional Analysis and Measure Theory. We mention the derived concept of ``Radon-Nikodym'' Derivative (which, in a certain sense, also generalises that of the ordinary derivative for smooth functions) which expresses a measure $\m$ dominated by another one $\n$ through the well-known formula $\di\m=\big[\frac{\di\m}{\di\n}\big]\di\n$. Here, the Radon Nikodym derivative is a positive $\n$-summable function (i.e. $[\frac{\di\m}{\di\n}\big]\in L^1(X,\n)_+$), and
one writes $\m\prec\n$. By looking at measurable sets $A\subset X$, $\m\prec\n$ simply means that 
\begin{equation}
\label{mnust}
\n(A)=0\Rightarrow \m(A)=0\,.
\end{equation}

Connected with this measure-theoretic notion of dominance, there is the notion of equivalence between measures, that is when $\m\prec\n\,\,\&\,\, \n\prec\m$. This means that both Boolean measure algebras associated to $\m$ and $\n$ coincide or, equivalently, the reciprocal function $1/\big[\frac{\di\m}{\di\n}\big]$ belongs to $L^1(X,\m)_+$ as well and provides the associated Radon-Nikodym derivative: 
$$
\Big[\frac{\di\n}{\di\m}\Big]=1/\Big[\frac{\di\m}{\di\n}\Big]\,,\quad a.\,e.\,\,\,.
$$
In order to pass to the noncommutative setting, note that, in this case, the underlying objects one is dealing with are (commutative) $C^*$-algebras $\ga$ made of bounded measurable functions on $X$ and, for any probability measure $\m$ on $X$, 
$$
\om_\m(f):=\int_X f(x)\di\m(x)\,,\quad f\in\ga\,.
$$
defines a state on $\ga$. The converse is also true by the Riesz-Markov Theorem, taking into account the Gelfand $*$-isomorphism 
$\ga\sim C_\infty\big(\s(\ga)\big)$.

The generalisation of such a notion of dominance and equivalence among the class of positive (always supposed linear without further mention) functionals to the noncommutative setting seems to be quite a delicate question, perhaps tackled in a very general framework in \cite{G}.  First of all, one is tempted to replace finite measures $\l$ on the Borel space $X$ with positive functionals $\psi$ on the $C^*$-algebra $\ga$ and, instead of computing quantities such as $\l(A)$ for measurable sets $A$, replace them with quantities like $\psi(a^*a)$ for $a\in\ga$. Since we can deal with $C^*$-algebras $\ga$, possibly abelian, which do not contain any projection but the trivial ones $0$ and $1$, the notion of dominance based on \eqref{mnust} can produce apparent paradoxes.
This is well explained in \cite{G} by some pivotal examples which we have reported here (cf. Examples \ref{exa1} and \ref{exa2}) for the convenience of the reader. The other difficulty is that, in the noncommutative case, the notion of ``point'' disappears, and thus we must reason at the level of algebras, dualising all the involved objects.

As explained in Subsection \ref{nodo}, it seems natural to start with the following notions:\\

\smallskip

\noindent
for $\om,\f$ positive linear functionals on the $C^*$-algebra $\ga$, we say that $\om$ is {\it strongly absolutely continuous} w.r.t. $\f$, and write $\om\prec\f$ if, for each sequence $(a_n)_n\subset\ga$ such that
$$
\lim_n\f(a_n^*a_n)=0=\lim_{m,n}\om\big((a_n-a_m)^*(a_n-a_m)\big)\,,
$$
it also follows that $\lim_n\om(a_n^*a_n)=0$.

\bigskip

\noindent
If $\om$ is strongly absolutely continuous w.r.t. $\f$ according to the previous definition, we write $\om\prec\f$ and, in the case of equivalence (i.e. 
$\om\prec\f\,\,\&\,\,\f\prec\om$), we write $\om\sim\f$.

\medskip

For the positive functional $\f$ on the $C^*$-algebra $\ga$, let $(\pi_\f,\ch_\f,\xi_\f)$ be the associated GNS representation.
As in similar previous papers (cf. \cite{N, G, Hi0, Hi}), there is a one-to-one correspondence between 
positive functionals $\om$ strongly absolutely continuous w.r.t. $\f$ and positive selfadjoint operators $T_{\om/\f}$ affiliated with the commutant algebra
$\pi_\f(\ga)'$, that is the {\it Radon-Nikodym Derivative}, such that $\pi_\f(\ga)\xi_\f$ is a core for $T_{\om/\f}^{1/2}$ and
$$
\om(a)=\big\langle\pi_\f(a)T_{\om/\f}^{1/2}\xi_\f,T_{\om/\f}^{1/2}\xi_\f\big\rangle\,,\quad a\in\ga\,.
$$
Notice that this result generalises the well-known one (cf. Theorem \ref{borarb})
relative to the bounded case where, for the Radon-Nikodym Derivative, $T_{\om/\f}\in\pi_\f(\ga)'$.

The (non) surprising fact is that ``$\prec$'', and also ``$\sim$'', are non-transitive. This pathology is indeed natural for the following reason: if $\om\prec\f$ and $\f\prec\psi$, the candidate for the  Radon-Nikodym Derivative describing the possible relation 
$\om\prec\psi$ should be recovered as the product $T^{1/2}_{\om/\f}T^{1/2}_{\f/\psi}$ (i.e. the {\it chain rule}, see Theorem \ref{notrans}).
But it is well known that the product of (unbounded) positive selfadjoint operators may well produce a non-closable operator, see Subsection \ref{maov}.

To end this preliminary part of the introduction, we recall that this notion of dominance is used in \cite{N} to generalise the notion of von Neumann Entropy in the setting of $W^*$-algebras, whereas it was used in \cite{Hi} in connected problems involving Quantum Information Geometry.

In the framework of the general theory of $W^*$-algebras, we mention the Pedersen-Takesaki construction in which there is also a kind of Radon-Nikodym Derivative satisfying nice cocycle properties. Such an unbounded operator lies (i.e. is affiliated)
in the algebra $\pi_\f(\ga)''$ or, more precisely, in the centraliser, generated by the GNS representation of a state with central support in the bidual. Therefore, such a construction has nothing to do with the (noncommutative generalisation of classical 
measure-theoretical) dominance between positive linear functionals used in the present paper, unless the Radon-Nikodym derivative is affiliated with the centre. Since this last case has natural applications in Quantum Statistical Mechanics (see e.g. Vol II of \cite{BR}), we also compare our framework of dominance with the Pedersen-Takesaki construction (cf. Section \ref{PTKMS}).

A natural application of the notion of dominance (or, more precisely, strong absolute continuity) and the relative equivalence in the measure-theoretical framework as explained before, is that in Ergodic Theory for noncommutative dynamical systems. Consider a $C^*$-dynamical system $(\ga, G,\a_g)$ made of a $C^*$-algebra $\ga$ and an action $G\stackrel{\a_g}{\curvearrowright}\ga$ of the group $G$ on $\ga$ by $*$-automorphisms. It is natural and customary to look at the (convex, ${}^*$-weakly compact) set $\cs_G(\ga)$ made of states which are invariant under such an action of $G$. If, on the one hand, such states have a fundamental importance, mainly for applications to Physics (e.g. \cite{BR, E}), on the other hand, to limit the analysis to invariant states seems to be too restrictive an assumption. The example we would like to mention is the early construction of type III von Neumann factors of the desired kind by using the crossed product, see e.g.  \cite{Kr}. Such a construction is based on $C^*$-dynamical systems exhibiting a quasi-invariant, non-invariant measure, singular w.r.t. the Haar-Lebesgue measure. Such a construction is not explicit because the existence of such a singular quasi-invariant measure is assured by general theorems. Recently, such an abstract construction was made explicit in \cite{FS} in order to construct non-type II${}_1$ representations of the irrational rotation algebra, together with the associated modular spectral triples. In light of the previous considerations, it is very natural to systematically address the study of dynamical systems exhibiting states which are merely quasi-invariant under the action of the involved group. 

The first attempt to investigate quasi-invariance was made in \cite{GK} using the notion of quasi-equivalence of the associated GNS representations. After having clarified the concept of dominance in the noncommutative setting and studied the basic properties (see Sections \ref{secmain}, \ref{secmain1}), it is then possible to carry out a fruitful investigation of such dynamical systems exhibiting quasi-invariant, non-invariant states.

The first fact we prove is that the pathology of non-transitivity does not appear in this situation: for a quasi-invariant state $\f$, we show that $\f\circ\a_g\sim\f\circ\a_h$ for all $g,h\in G$ (cf. Proposition \ref{ortiiq}). As a consequence, we can provide the corresponding ``chain rule'' associated to the Radon-Nikodym Derivative (cf. Theorem \ref{lenire}). The following (not immediate to prove) relevant facts hold true. First, the set of quasi-invariant states is closed under convex combinations (Propositions \ref{concon}, \ref{concon1}), but not topologically closed unlike that made up of the invariant states. Second, we can exhibit a canonical implementation (Theorem \ref{twogns} and Remark \ref{lenire0}),
which is, in general, not a representation, of the action of the involved group in the GNS representation of the involved quasi-invariant state, generalising the well-known covariant GNS representation associated to an invariant one.

The present paper is complemented by a section (Section \ref{lenire00}) describing an unbounded implementation under which the Radon-Nikodym Derivative satisfies a nice cocycle property. We end with some examples helping to clarify the situation arising from quasi-invariance.

\section{preliminaries}
\label{ptmgv}

We gather here some basic facts useful in the sequel. For more general results, we refer the reader to standard textbooks of Functional Analysis and Operator Theory.
\subsection{Basic facts} 
For a Hilbert space, always built over the complex numbers, we denote by $\langle\,\,\,,\,\,\,\rangle$ its inner product, which is supposed to be linear w.r.t. the 1st variable. With $\cb(\ch)$ and $\cu(\ch)$, we denote the von Neumann algebra consisting of all linear bounded operators, and the group of unitary operators acting on $\ch$, respectively. In the present paper, all operators, as well as functionals, are supposed linear (except for the Tomita involution and conjugation coming from Tomita theory, which are antilinear) without any further mention.

For a unital involutive algebra $\gc$ with a unique selfadjoint unit $\idd_{\gc}$, we denote it as $1\equiv\idd_{\gc}$ if it causes no confusion. The group of all linear unitary operators (i.e. those satisfying $u^*u=1=uu^*$) is denoted by $\cu(\gc)$. Concerning  a $C^*$-algebra $\ga$, we tacitly suppose that it is unital, if needed, after pointing out that most of the results in the present paper hold true without assuming unitality.

For the convenience of the reader, we report the following basic result.
\begin{lem}
\label{piuno}
Let $\mathcal H$ be a Hilbert space, $W:\mathcal D(W)\subset \mathcal H\rightarrow \mathcal H$ be a densely defined closed operator. Let $\mathcal D$ be a core for $W$. Then
$$
\overline{\mathcal R(W\lceil_\mathcal D)}=\overline{\mathcal R(W)}\,.
$$ 
\end{lem}
\begin{proof}
 Let $W_0=W\lceil_\mathcal D$. Since $W$ is closed and $\mathcal D$ is a core for $W$, by Theorem 3.2.3 in \cite{D}, one gets
$\overline{\mathcal R(W_0)}=\mbox{Ker}(W_0^*)^\bot=\mbox{Ker}(W^*)^\bot =\overline{\mathcal R(W)}$.
\end{proof}

For a projection in a $C^*$-algebra, we always mean a selfadjoint one. For the $C^*$-algebra (and, in particular, for a $W^*$-algebra or a von Neumann algebra) $\ga$, with $\cp(\ga)$ and $\cu(\ga)$ we denote the set of all selfadjoint projections and the group of all unitaries in $\ga$ (provided $\ga$ is unital), respectively. 
Notice that the set of the projections of a $C^*$-algebra might consist only of $0$ (and $1$ in the unital case), whereas it generates any $W^*$-algebra. If $\ga=\cb(\ch)$, the von Neumann algebra of all bounded linear operators acting on the Hilbert space $\ch$, the above sets will be simply denoted as $\cp(\ch)$ and $\cu(\ch)$.

Let $\ga$ be a $C^*$-algebra, $\f\in\cs(\ga)$ (or, merely, a positive linear functional on $\ga$), and $\a\in{\rm aut}(\ga)$, the group of all $*$-automorphisms of $\ga$. We denote by $\f_\a:=\f\circ\a$.

For an {\it action} $G\stackrel{\a_g}{\curvearrowright}\ga$ of any group $G$ on a $C^*$-algebra $\ga$, we mean a representation
$$
G\ni g\mapsto\a_g\in{\rm aut}(\ga)
$$
of $G$ into the group ${\rm aut}(\ga)$. For $\f$ being a positive linear functional, we put $\f_g:=\f\circ\a_g\equiv\f_{\a_g}$.

For a {\it unitary representation} $g\to U_g$, we simply mean a representation of $G$ into $\cu(\ch)$.

Let $(M,\ch)$ be a von Neumann algebra.
For $\xi\in\ch$, with $\upsilon_\xi$, we denote the restriction to $M$ of the vector functional associated to $\xi$:
$$
\upsilon_\xi(A):=\langle A\xi,\xi\rangle\,,\quad A\in M\,.
$$
Obviously, $\upsilon_\xi$ depends on the von Neumann algebra $M\subset\cb(\ch)$. For the sake of simplicity, we omit to indicate such a dependence if it causes no confusion. We also say that a vector $\xi\in\ch$ is {\it standard} if it is cyclic and separating for $M$, and then also for $M'$: it means nothing else than $\overline{M\xi}=\ch=\overline{M'\xi}$.

A closed operator $A$ acting on $\ch$, with $A=V|A|$ its polar decomposition, and $|A|=\int\l\di E_{|A|}(\l)$ its resolution, is said to be {\it affiliated} with $M$ (written $A\,\eta\,M$) if, by definition, $\big\{V, (E_{|A|}(\l))_{\l\in\br}\big\}\subset M$. It is well known that
$$
A\,\eta\,M\iff A\,\text{commutes with all unitaries}\,U'\in M'
$$
(i.e. $U'A(U')^*=A$ for all $U'\in\cu(M')$), see e.g. \cite{SZ}, E.9.25.

\bigskip

Let $\ga$ be a $C^*$-algebra, and consider a state (or, merely, a positive linear functional) $\f$ on $\ga$. Its Gelfand-Naimark-Segal (GNS for short) representation is denoted by $(\pi_\f,\ch_\f,\xi_\f)$. 

For such a $C^*$-algebra $\ga$, we consider its canonical injection $\k:\ga\to\ga^{**}$ in the enveloping von Neumann algebra, which is a $*$-monomorphism. For each $\f\in\ga^*$, the relation
\begin{equation}
\label{ssbdh}
\iota(\f)(\k(a))=\f(a)\,,\,\,a\in\ga\,,
\end{equation}
realises the identification $\iota:\ga^*\to(\ga^{**})_*$. We say that $\f\in\cs(\ga)$ has {\it central support in the bidual} if the support 
$s(\iota(\f))\in\cp(\ga^{**})$ of its normal extension in the bidual is central: $s(\iota(\f))\in Z(\ga^{**})$.

\subsection{Some results about Modular Theory} 
For the basic facts about Tomita Theory (also known as {\it Modular Theory} or Tomita-Takesaki Theory) for $W^*$-algebras, the reader is referred to \cite{BR, H, S, SZ}.

Let $\gam$ be a $W^*$-algebra. In order to exhibit a (unique up to unitary equivalence) {\it standard form} for it, it is enough to fix a normal semifinite faithful (n.s.f. for short) weight $\psi$ (which always exists) and construct the GNS representation $(\pi_\psi, \ch_\psi)$ relative to $\psi$. In this situation, $S_\psi$, $J_\psi$ and $\D_\psi$ will denote {\it Tomita's involution}, {\it Tomita's conjugation} and the {\it Modular operator} relative to $\psi$.

For the n.s.f. weight $\psi$, we can identify
$\gam$ with $M:=\pi_\psi(\gam)$, and reduce the matter to the standard form $(M, L^2(M), J, \cp)$ (unique, up to unitary equivalence), where $L^2(M)=\ch_\psi$, $J:=J_\psi$ and, finally, the self-dual cone $\cp$ is that associated with $\psi$, see \cite{H}. 
Therefore, each normal positive functional $\f$ of a von Neumann algebra in standard form is represented as a vector functional 
$\upsilon_{\xi}$ for some $\xi\in L^2(M)$ as described above. Such a choice of the representative vector $\xi$ is unique if one requires that $\xi\in\cp$, the canonical cone associated to the standard form.

We briefly describe the case of a type $\ty{I}$ factor, that is $M=\cb(\ell^2(\Lambda))$, $\Lambda$ being any index-set. On $\ell^2(\Lambda)$ and $\cb(\ell^2(\Lambda))$, we consider the canonical basis $\{e_i\}_{i\in \Lambda}$, and the canonical system of matrix-units $\{e_{ij}\}_{i,j\in \Lambda}$. Any elements $\xi\in\ell^2(\Lambda)$ and $X\in\cb(\ell^2(\Lambda))$ have the expansions
$\xi=\sum_{i\in \Lambda}x_ie_i$, $X=\sum_{i,j\in \Lambda}X_{ij}e_{ij}$, respectively.
In such a situation, $\cb(\ell^2(\Lambda))$ is represented in standard form on $\ch:=L^2(\cb(\ell^2(\Lambda)))$ consisting of all Hilbert-Schmidt linear operators. Indeed, for 
$\xi=\sum_{i,j\in \Lambda}\xi_{ij}e_{ij}$, $\eta=\sum_{i,j\in \Lambda}\eta_{ij}e_{ij}$,
elements of $L^2(\cb(\ell^2(\Lambda)))$,
the inner product is given by
$$
L^2(\cb(\ell^2(\Lambda))\ni \xi,\eta\mapsto \langle \xi\,,\eta\rangle:={\rm Tr}(\eta^*\xi)=\sum_{i,j\in \Lambda} \xi_{ij}\overline{\eta_{ij}}\in\bc\,.
$$
 
For $X\in\cb(\ell^2(\Lambda))$ and $\xi\in L^2(\cb(\ell^2(\Lambda)))$ , the {\it left} and {\it right representations} of $\cb(\ell^2(\Lambda))$  are defined as follows:
\begin{equation}
\label{anakd}
\cb(\ell^2(\Lambda))\ni X \mapsto\left\{\begin{array}{ll}
                     \!\!\!\!\!\!\!&(\pi_{\rm L}(X)\xi)_{ij}:=\sum_{k\in \Lambda} X_{ik}\xi_{kj} \\[1ex]
\!\!\!\!\!\!\!&(\pi_{\rm R}(X)\xi)_{ij}:=\sum_{k\in \Lambda} \xi_{ik}X_{kj}
                    \end{array}
                    \right\}\in L^2(\cb(\ell^2(\Lambda))\,.
\end{equation} 
Notice that $\pi_{\rm L}\big(\cb(\ell^2(\Lambda))\big)'=\pi_{\rm R}\big(\cb(\ell^2(\Lambda))\big)$, and vice-versa.

Fix now a double sequence $(A_{ij})_{i,j\in \Lambda}$ and define, as in \eqref{anakd},
$$
\pi_{\rm L}(A)\xi:=\sum_{i,j\in \Lambda}\Big(\sum_{k\in \Lambda}A_{ik}\xi_{kj}\Big)e_{ij}\,,\quad
\pi_{\rm R}(A)\xi:=\sum_{i,j\in \Lambda}\Big(\sum_{k\in \Lambda}\xi_{ik}A_{kj}\Big)e_{ij}\,,
$$
on the domains
\begin{align*}
&\cd({\pi_{\rm L}}(A)):=\big\{\xi\in L^2(\cb(\ell^2(\Lambda)))\,;\,\|\pi_{\rm L}(A)\xi\|<+\infty\big\}\,,\\
&\cd({\pi_{\rm R}}(A)):=\big\{\xi\in L^2(\cb(\ell^2(\Lambda)))\,;\,\|\pi_{\rm R}(A)\xi\|<+\infty\big\}\,.
\end{align*}

The proof of the following result is quite standard and is left to the reader.
\begin{prop}
\label{poseaf}
With the above notations, $\pi_{\rm L}(A)$, $\pi_{\rm R}(A)$ are densely defined closed linear operators and, in addition,
\begin{align*}
&\pi_{\rm L}(A)\,\eta\,\pi_{\rm L}(\cb(\ell^2(\Lambda)))=\pi_{\rm R}(\cb(\ell^2(\Lambda)))'\,,\\
&\pi_{\rm R}(A)\,\eta\,\pi_{\rm R}(\cb(\ell^2(\Lambda)))=\pi_{\rm L}(\cb(\ell^2(\Lambda)))'\,.
\end{align*}
\end{prop}

\medskip

We now pass to the following result, quite well known to the experts. Here, we sketch the proof for the convenience of the reader.
\begin{thm}
\label{zscsu}
Let $\ga$ be a $C^*$-algebra, and $\f\in\cs(\ga)$. Then $\f$ has central support in the bidual if and only if $\xi_\f$ is a standard vector.
\end{thm}
\begin{proof}
Notice that, for $\pi_\f^{**}=\pi_{\iota(\f)}$ ($\iota(\f)$ given in \eqref{ssbdh}) we have 
\begin{equation}
\label{ssbdh2}
\pi_\f^{**}(\ga^{**})=\pi_{\iota(\f)}(\ga^{**})=\pi_\f(\ga)''\subset\cb(\ch_\f)
\end{equation}
since $\iota(\f)$ is normal. Therefore, for the (weak${}^*$-closed) left ideal associated to $\iota(\f)$ and
the kernel of $\pi_{\iota(\f)}$ given, respectively, by
\begin{align*}
\gpn_{\iota(\f)}:=&\{x\in\ga^{**}\,;\,\iota(\f)(x^*x)=0\}\,,\\ 
{\rm Ker}(\pi_{\iota(\f)}):=&\{x\in\ga^{**}\,;\,\pi_{\iota(\f)}(x)=0\}\,,
\end{align*}
we have
\begin{equation}
\label{ssbdh1}
\ga^{**}\big(1-z(\iota(\f))\big)={\rm Ker}(\pi_{\iota(\f)})\subset\gpn_{\iota(\f)}=\ga^{**}\big(1-s(\iota(\f))\big)\,.
\end{equation}

The property of $\xi_\f$ being standard holds if and only if $\upsilon_{\xi_\f}$ is faithful, which is equivalent to the fact that
${\rm Ker}(\pi_{\iota(\f)})$ coincides with $\gpn_{\iota(\f)}$ by \eqref{ssbdh2}. By \eqref{ssbdh1}, the last property is equivalent to
the fact that the support $s(\iota(\f))$ of $\iota(\f)$ coincides with its central support $z(\iota(\f))$. Therefore, the property of $\f$ having central support is equivalent to the property of $\xi_\f$ being standard.
\end{proof}

For $\f\in\cs(\ga)$, on the one hand (by \eqref{ssbdh2})
\begin{equation}
\label{ssbdh3}
\pi_\f(\ga)''\sim\ga^{**}z(\iota(\f))\,.
\end{equation}
If $\f$ has central support, $z(\iota(\f))=s(\iota(\f))$ and the state $\iota(\f)$ is a KMS state w.r.t. its own modular group 
$\big(\s_t^{\iota(\f)}\big)_{t\in\br}$, which acts identically on $\ga^{**}\big(1-s(\iota(\f))\big)$, see e.g. \cite{BR}, Theorem 5.3.10. On the other hand, since $\xi_\f$ is a standard vector for $\pi_\f(\ga)''$, $\upsilon_{\xi_\f}$ (restricted to $\pi_\f(\ga)''$) is a KMS state w.r.t. its own modular group $\big(\s_t^{\upsilon_{\xi_\f}}\big)_{t\in\br}$ as well, e.g. \cite{BR}, Section 2.5.
\begin{rem}
Under the identification \eqref{ssbdh3}, we have 
$$
\s_t^{\upsilon_{\xi_\f}}\sim\s_t^{\iota(\f)}\lceil_{\ga^{**}s(\iota(\f))}\,,\quad t\in\br\,.
$$
\end{rem}
Summarising, for a state $\f$ on the $C^*$-algebra $\ga$ with central support, we can directly consider the generated von Neumann algebra 
$(\pi_\f(\ga)'',\ch_\f)$. In this situation, for the associated modular operators, we simply  write $S_\f\equiv S_{\upsilon_{\xi_\f}}$, $J_\f\equiv J_{\upsilon_{\xi_\f}}$ and 
$\D_\f\equiv \D_{\upsilon_{\xi_\f}}$.

\subsection{On various notions of dominance} 
\label{nodo}
We report a preliminary analysis extracted from the seminal paper \cite{G}, explaining the motivations for assuming the definition of dominance in the present paper.

Let $\ga$ be an involutive algebra (no topology is assumed at the moment, but for the rest of the present paper we indeed deal with $C^*$-algebras) and $\f$ and $\om$ two positive (always linear without any further mention) functionals. 

A sequence 
$(a_n)_n\subset\ga$ is called a $(\f,\om)${\it -sequence} if
$$
\lim_n\f(a_n^*a_n)=0=\lim_{m,n}\om\big((a_n-a_m)^*(a_n-a_m)\big)\,.
$$
\begin{defin}
\label{dom}
For $\om,\f$ positive functionals, we say that 
\begin{itemize}
\item[(i)] $\om$ is $\f$-{\it dominated} if there exists $M>0$ such that
$$
\om(a^*a)\leq M\f(a^*a)\,,\quad a\in\ga\,;
$$
\item[(ii)] $\om$ is {\it strongly $\f$-absolutely continuous} if, for each $(\f,\om)$-sequence $(a_n)_n\subset\ga$, it follows that
$\lim_n\om(a_n^*a_n)=0$;
\item[(iii)] $\om$ is {\it $\f$-absolutely continuous} if
$$
\f(a^*a)=0\Rightarrow\om(a^*a)=0\,,\quad a\in\ga\,.
$$
\end{itemize}
\end{defin}
For the convenience of the reader, we report the following result (\cite{BR}, Theorem 2.3.19), generalising the well-known situation in measure theory to the noncommutative case.
\begin{thm}
\label{borarb}
For a $C^*$-algebra $\ga$ and $\f$ a fixed positive linear functional with $(\pi_\f,\ch_\f,\xi_\f)$ its GNS representation, there is a one-to-one correspondence between positive $\f$-dominated functionals $\om$, and positive elements $T^{1/2}_{\om/\f}\in\pi_\f(\ga)'$ such that
$$
\om(a)=\langle\pi_\f(a)T^{1/2}_{\om/\f}\xi_\f,T^{1/2}_{\om/\f}\xi_\f\rangle\,,\quad a\in\ga\,.
$$
\end{thm}
Also the following examples, reported from \cite{G}, are in order.
\begin{exa}
\label{exa1}
We consider a measurable space $(X,\S)$ equipped with two probability measures $\m$ and $\n$. Let $\ga$ be the $C^*$-algebra made of all bounded measurable functions w.r.t. the sup-norm, and define
$$
\om(f):=\int_X fd\m,\,\,\f(f):=\int_X fd\n,\quad f\in\ga\,.
$$
\end{exa}
By using the Radon-Nikodym Theorem, it is quite easy to see that, in Example \ref{exa1}:
\begin{itemize}
\item[$1^o$] $\om$ is {\it $\f$-absolutely continuous}$\iff \m\prec\n$ in the measure-theoretical setting;
\item[$2^o$] by using the classical Radon-Nikodym Theorem, it is shown that
$$
\m\prec\n\Rightarrow\om\,\,\text{is strongly $\f$-absolutely continuous}\,. 
$$
\end{itemize}
Therefore, (ii) is equivalent to (iii) in this situation. On the other hand, it is well known that (iii) does not imply (i).
\begin{exa}
\label{exa2}
Let $\ga=C([0,1])$, with $\di x$ denoting the Lebesgue measure. For $f\in C([0,1])$, put
\begin{equation*}
\begin{split}
&\f(f):=\int_0^1 f(x)dx\,,\\
&\om(f)=f(x_o)\,\,\text{for each fixed}\,\, x_o\in[0,1]\,.
\end{split}
\end{equation*}
\end{exa}
In Example \ref{exa2}, it is almost immediate to see that $\om$ is $\f$-absolutely continuous but not strongly $\f$-absolutely continuous.

Summarising, we first note that, clearly (i)$\Rightarrow$(ii)$\Rightarrow$(iii). The previous examples lead to (iii)$\nRightarrow$(ii)$\nRightarrow$(i) and, on the other hand, (iii) is not the right notion of dominance (in Example \ref{exa2}, $\om$ is $\f$-absolutely continuous but the underlying measures are mutually singular). Therefore, we will adopt (ii) in Definition \ref{dom} for the extension of the {\it dominance} to the general cases including, mainly, the noncommutative ones.

\section{The dominance relation and the Radon-Nikodym derivative}
\label{secmain}

Let $\ga$ be a $C^*$-algebra and $\om,\f\in\ga^*_+$ two positive functionals. In the present paper, we assume (ii) in Definition \ref{dom} as the definition of dominance and, if this is the case, that is, if $\om$ is strongly $\f$-absolutely continuous, we write $\om\prec\f$. If, also, $\f\prec\om$, we write $\om\sim\f$.
We will see later that ``$\prec$'' is not transitive because transitivity is equivalent to the closability of the product of the corresponding ``Radon-Nikodym'' derivatives, which is, of course, not always granted. Consequently, it is not a partial order relation among the set of the positive linear functionals on $\ga$. As a consequence, ``$\sim$'' is not an equivalence relation. We will also see that, for states which are quasi-invariant under the action of a group, such a ``pathology'' disappears.

The following result is the adaptation to our situation of the previous analogous results (e.g. \cite{N}, Theorem 2.3; \cite{G}, Theorem 1; \cite{Hi}, Lemma 3.1).
\begin{thm}
\label{main}
Let $\ga$ be a $C^*$-algebra and $\f,\om\in\ga^*_+$. The following assertions are equivalent:
\begin{enumerate} 
\item[(i)] $\om\prec\f$;
\item[(ii)] there exists a positive self-adjoint operator $T_{\om/\f}$, affiliated with $\pi_\f(\ga)'$, such that 
$\pi_\f(\ga)\xi_\f$ is a core for $T_{\om/\f}^{1/2}$ and
$$
\om(x)=\big\langle T_{\om/\f}^{1/2}\pi_\f(x)\xi_\f,T_{\om/\f}^{1/2}\xi_\f\big\rangle\,,\quad x\in\ga\,.
$$
\end{enumerate}
If one of the above equivalent conditions holds true and, in particular, (ii), the operator $T_{\om/\f}$ is uniquely determined.
\end{thm}
\begin{proof}
We drop the subscripts to simplify notations.

(i)$\Rightarrow$(ii) On $D:=\pi_\f(\ga)\xi_\f$ we put
$$
\pi_\f(\ga)\xi_\f\ni\pi_\f(a)\xi_\f\mapsto R_o\pi_\f(a)\xi_\f:=\pi_\om(a)\xi_\om\,.
$$
First we notice that $R_o$ is well defined and closable on $D$. 
Indeed, suppose that $\pi_\f(a)\xi_\f=0$ and consider the stationary sequence $a_n=a$. Then
\begin{align*}
0=&\|\pi_\f(a)\xi_\f\|^2=\f(a^*a)=\lim_n\f(a_n^*a_n)\Rightarrow\lim_n\om(a_n^*a_n)\\
=&\om(a^*a)=\|\pi_\om(a)\xi_\om\|^2=0\,,
\end{align*}
and thus $\pi_\f(a)\xi_\f=0\Rightarrow\pi_\om(a)\xi_\om=0$.
Passing to closability, with $\xi_n:=\pi_\f(x_n)\xi_\f$ (i) means
$$
\lim_n\|\xi_n\|^2=0=\lim_{m,n}\|R_o\xi_m-R_o\xi_n\|^2\Rightarrow\lim_n\|R_o\xi_n\|^2=0\,.
$$
But this is nothing else than the closability of $R_o$ on its core $D$.

Now we follow the proof of the analogous Lemma 3.1 in \cite{Hi}. Denote $R:=\overline{R_o}$, with polar decomposition $R=VT^{1/2}$. Now, if $u\in\mathcal U(\ga)$, we get
\begin{align*}
\pi_\om(u)R_o\pi_\f(u^*)\pi_\f(a)\xi_\f=&\pi_\om(u)R_o\pi_\f(u^*a)\xi_\f=\pi_\om(u)\pi_\om(u^*a)\xi_\om\\
=&\pi_\om(a)\xi_\om=R_o\pi_\f(a)\xi_\f\,,
\end{align*}
so $\pi_\om(u)R_o\pi_\f(u^*)=R_o$.
Therefore, by passing to the closures, 
$$
\pi_\om(u)R\pi_\f(u^*)=R\,,\quad u\in\cu(\ga)\,.
$$
By considering $T=R^*R$, we obtain
\begin{align*}
T=&R^*R=\pi_\f(u)R^*\pi_\om(u^*)\pi_\om(u)R\pi_\f(u^*)\\
=&\pi_\f(u)R^*R\pi_\f(u^*)=\pi_\f(u)T\pi_\f(u^*)\,.
\end{align*}
Since $T$ commutes with all unitaries of the form $U=\pi_\f(u)$, $u$ unitary in $\ga$, $T$, and thus $T^{1/2}$ is affiliated with $\pi_\f(\ga)'$.

By definition, $\pi_\f(\ga)\xi_\f$ is a core for $R=\overline{R_o}$, and thus it is also a core for $T^{1/2}$. Finally,
\begin{align*}
\om(x)=&\langle\pi_\om(x)\xi_\om,\xi_\om\rangle=\langle R\pi_\f(x)\xi_\f,R\xi_\f\rangle=\langle VT^{1/2}\pi_\f(x)\xi_\f,VT^{1/2}\xi_\f\rangle\\
=&\langle V^*VT^{1/2}\pi_\f(x)\xi_\f,T^{1/2}\xi_\f\rangle=\langle T^{1/2}\pi_\f(x)\xi_\f,T^{1/2}\xi_\f\big\rangle\,.
\end{align*}

(ii)$\Rightarrow$ (i). Let $(x_n)_n\subset\ga$ be a $(\f,\om)$-sequence. Notice that
$$\om(x_n^*x_n)=||T_{\om/\f}^{1/2}\pi_\f(x_n)\xi_\f||^2. $$
Since $(x_n)_n\subset\ga$ is a $(\f,\om)$-sequence, one has
$$\lim_n\varphi(x_n^*x_n)=||\pi_\f(x_n)\xi_\f||^2=0,$$
which implies that $\lim_n\pi_\f(x_n)\xi_\f=0$. Moreover, one has
$$0=\lim_{n,m}\om((x_n-x_m)^*(x_n-x_m))=||T_{\om/\f}^{1/2}\pi_\f(x_n)\xi_\f-T_{\om/\f}^{1/2}\pi_\f(x_m)\xi_\f||^2.$$
It follows that $\left( T_{\om/\f}^{1/2}\pi_\f(x_n)\xi_\f\right)_n$ is Cauchy. Therefore $\left( T_{\om/\f}^{1/2}\pi_\f(x_n)\xi_\f\right)_n$ converges to $y\in\mathcal H_\varphi$. Since $T_{\om/\f}^{1/2}$ is closed, one gets $y=0$ and hence
$$
\lim_n\om(x_n^*x_n)=\lim_n||T_{\om/\f}^{1/2}\pi_\f(x_n)\xi_\f||^2=0\,.
$$

Concerning the uniqueness of the Radon-Nikodym derivative, we reason as in \cite{N}. Let $S$ and $T$ be two such density operators associated to the relation $\om\prec\f$. On the common core $\pi_\f(\ga)\xi_\f$ for $S^{1/2}$ and $T^{1/2}$, we have
 $$
 \|S^{1/2}\pi_\f(a)\xi_\f\|^2=\om(a^*a)= \|T^{1/2}\pi_\f(a)\xi_\f\|^2\,,\quad a\in\ga\,.
 $$
Therefore, there exists a partial isometry $U$ from the closure of the range of $S^{1/2}$ to the closure of the range of $T^{1/2}$
satisfying
$$
 US^{1/2}\pi_\f(a)\xi_\f=T^{1/2}\pi_\f(a)\xi_\f\\,,\quad a\in\ga\,.
$$
We then deduce that $US^{1/2}\supset T^{1/2}$ and, reversing the role of $S$ and $T$, $U^*T^{1/2}\supset S^{1/2}$. But this implies $US^{1/2}=T^{1/2}$. Since the polar decomposition is unique, we conclude that $U$ is the selfadjoint projection onto the closure of the range of $S^{1/2}$ and $S^{1/2}=T^{1/2}$.
\end{proof}

The closed operators $T_{\om/\f}$ and $T^{1/2}_{\om/\f}$ are called the $L^1$ and, respectively, $L^2$-{\it Radon-Nikodym} derivative of the dominated state $\om$ w.r.t. the dominant state $\f$.

At this stage we note the perfect analogy with Theorem \ref{borarb} which establishes a one-to-one correspondence between the $\f$-dominated positive functionals and positive elements in the commutant algebra $\pi_\f(\ga)'$. This corresponds to the bounded situation. As explained by the previous theorem, the unbounded situation, for which the Radon-Nikodym derivative is unbounded but still related (i.e. affiliated) with the commutant, corresponds to strong $\f$-absolute continuity. We also mention the situation treated in \cite{PT} involving states $\f$ having central support in the bidual. In such a situation, the Radon Nikodym derivative is assumed to be in the algebra $\pi_\f(\ga)''$ (more precisely in the centraliser). Such a situation has nothing to do with the noncommutative generalisation of the measure-theoretical notion of dominance, unless the Radon Nikodym derivative belongs to the centre
$\pi_\f(\ga)''\bigcap\pi_\f(\ga)'$. The last case has relevant applications in Physics and, mainly, to Quantum Statistical Mechanics. The forthcoming Section \ref{PTKMS} is entirely devoted to this aspect.

We are mainly interested in states on $\ga$, without loss of generality, as the (unessential) normalisation can be disregarded. Therefore, from now on, we restrict the analysis to states unless otherwise specified.

The previous result allows us to view the GNS representation of the dominated state $\om$ directly as an induced subrepresentation of that of the dominating state $\f$. Indeed,
\begin{prop}
\label{pi}
Let $\omega,\varphi\in\cs(\ga)$ such that $\omega\prec\varphi$. With $P_{\om/\f}$ denoting the projection onto the closed subspace $[\pi_\f(\ga)T^{1/2}_{\om/\f}\xi_\f]$, $P_{\om/\f}\in\pi_\f(\ga)'$ and 
$(P_{\om/\f}\pi_\f,  P_{\om/\f}\ch_\f, T^{1/2}_{\om/\f}\xi_\f)$ is a GNS representation of $\om$.
\end{prop}
\begin{proof}
Put $P:=P_{\om/\f}$ and $W:=T^{1/2}_{\om/\f}$. Since $\pi_\f(a)\pi_\f(\ga)W\xi_\f\subset\pi_\f(\ga)W\xi_\f$, $[\pi_\f(\ga)W\xi_\f]$ is an invariant subspace for $\pi_\f$, and thus $P\in\pi_\f(\ga)'$. Now, $W\xi_\f\in P\ch_\f$ is obviously cyclic (in $P\ch_\f$) for 
$P\pi_\f$ and
$$
\langle P\pi_\f(a)W\xi_\f,W\xi_\f\rangle=\langle\pi_\f(a)W\xi_\f,W\xi_\f\rangle=\om(a)\,,\quad a\in\ga\,.
$$
This concludes the proof.
\end{proof}
\begin{rem}
Let $\om\prec\f\in\cs(\ga)$. Thanks to Proposition \ref{pi}, we can (and we always do, unless otherwise specified) identify the GNS representation $(\pi_\om,\ch_\om,\xi_\om)$ of the state 
$\om$ with
$(P_{\om/\f}\pi_\f,P_{\om/\f}\ch_\f, T^{1/2}_{\om/\f}\xi_\f)$.
\end{rem}

If $\om\prec\f$ then, obviously, $\l\om\prec\f$ for each positive $\l$ with $T_{(\l\om)/\f}=\l T_{\om/\f}$. The set of positive functionals 
$\om$, strongly absolutely continuous w.r.t. $\f$, is also closed under the sum. Indeed,
\begin{prop}
\label{concon}
For a $C^*$-algebra $\ga$ and a functional $\f\in\ga^*_+$, the set of $\om\in\ga^*_+$ with
$\om\prec\f$ is a convex cone. For $\om_1,\om_2\prec\f$, we have
$$
T_{(\om_1+\om_2)/\f}=T_{\om_1/\f}\dotplus T_{\om_2/\f}\,,
$$
where ``$\dotplus$'' denotes the form sum (e.g. \cite{S}, A11).
\end{prop}
\begin{proof}
We sketch the proof and leave the details to the reader.

Put $\om:=\om_1+\om_2$. Notice that if $(a_n)_n\subset\ga$ is a $(\f,\om)$-sequence, it is also a $(\f,\om_i)$-sequence, $i=1,2$.
But then, since $\om_1,\om_2\prec\f$, we have
\begin{align*}
\lim_n\om(a_n^*a_n)=&\lim_n\big(\om_1(a_n^*a_n)+\om_2(a_n^*a_n)\big)\\
=&\lim_n\om_1(a_n^*a_n)+\lim_n\om_2(a_n^*a_n)=0\,,
\end{align*}
and thus $\om_1+\om_2\prec\f$.

Notice that the Radon-Nikodym derivatives $T_{\om_i/\f}$, $i=1,2$ belong to the extended positive part $\overline{\pi_\f(\ga)'_+}$ (which is closed under the form sum, see e.g. \cite{S}), with Haagerup's notation
$$
T_{\om_i/\f}=\int_{0^-}^{+\infty}\l\di e_i(\l)+\infty(1-e_{i,\infty})\,,
$$ 
$i=1,2$.  Since both $T_{\om_i/\f}^{1/2}$ have the common core $D=\pi_\f(\ga)\xi_\f$, we first deduce $e_{i,\infty}=1$. In addition,
for 
$$
T_{\om_1/\f}\dotplus T_{\om_2/\f}=:T=\int_{0^-}^{+\infty}\l\di e(\l)+\infty(1-e_{\infty})\,,
$$
also $e_{\infty}=1$ and $D$ is still a core for $T^{1/2}$. Therefore, $T^{1/2}$ is a positive selfadjoint operator affiliated with $\pi_\f(\ga)'$ having
$D$ as a core. Finally, since $T$ is the form sum of $T_{\om_1/\f}$ and $T_{\om_2/\f}$, we get for $a\in\ga$,
\begin{align*}
&\langle\pi_\f(a^*a)T^{1/2}\xi_\f,T^{1/2}\xi_\f\rangle=\langle\pi_\f(a)T^{1/2}\xi_\f,\pi_\f(a)T^{1/2}\xi_\f\rangle\\
=&\langle T^{1/2}\pi_\f(a)\xi_\f,T^{1/2}\pi_\f(a)\xi_\f\rangle\,\,\text{(form sum property)}\\
=&\langle T_{\om_1/\f}^{1/2}\pi_\f(a)\xi_\f,T_{\om_1/\f}^{1/2}\pi_\f(a)\xi_\f\rangle
+\langle T_{\om_2/\f}^{1/2}\pi_\f(a)\xi_\f,T_{\om_2/\f}^{1/2}\pi_\f(a)\xi_\f\rangle\\
=&\langle \pi_\f(a)T_{\om_1/\f}^{1/2}\xi_\f,\pi_\f(a)T_{\om_1/\f}^{1/2}\xi_\f\rangle
+\langle \pi_\f(a)T_{\om_2/\f}^{1/2}\xi_\f,\pi_\f(a)T_{\om_2/\f}^{1/2}\xi_\f\rangle\\
=&\langle \pi_\f(a^*a)T_{\om_1/\f}^{1/2}\xi_\f,T_{\om_1/\f}^{1/2}\xi_\f\rangle
+\langle \pi_\f(a^*a)T_{\om_2/\f}^{1/2}\xi_\f,T_{\om_2/\f}^{1/2}\xi_\f\rangle\\
=&\om_1(a^*a)+\om_2(a^*a)\,.
\end{align*}
The assertion now follows by taking into account that any element of a $C^*$-algebra is a linear combination of positive ones, and by the uniqueness of the Radon-Nikodym derivative (cf. Theorem \ref{main}).
\end{proof}
\begin{rem}
Since $\om_i\sim\f$ is equivalent to requiring that $T_{\om_i/\f}$ is injective, $i=1,2$ (cf. Theorem \ref{invinv} below), in this situation it is possible to show that $T_{\om_1/\f}\dotplus T_{\om_2/\f}$ is also injective, and thus $\om_1+\om_2\sim\f$ as well.
\end{rem}

\section{The main properties of the Radon-Nikodym derivative}
\label{secmain1}

The following result characterises states which are equivalent in terms of properties of the corresponding Radon-Nikodym derivatives.
\begin{thm}
\label{invinv}
Let $\ga$ be a $C^*$-algebra and $\f,\om\in\cs(\ga)$ with $\om\prec\f$. 
With the notations of Proposition \ref{pi}, the following are equivalent:
\begin{itemize}
\item[(i)] $T_{\om/\f}^{1/2}\xi_\f$ is cyclic for $\pi_\f(\ga)$;
\item[(ii)] $T_{\om/\f}$ is injective;
\item[(iii)] $P_{\om/\f}=I_{\mathcal H_\varphi}$;
\item[(iv)]  $\f\prec\om$.
\end{itemize}
If one of the above conditions holds true and, in particular (ii),
we have $T_{\f/\om}=T_{\om/\f}^{-1}$.
\end{thm}
\begin{proof}
We start by noticing that the equivalence (i) $\Leftrightarrow$ (iii) is obvious.

(i) $\Rightarrow$ (ii) Since $T^{1/2}_{\om/\f}$ is selfadjoint,  $T^{1/2}_{\om/\f}\xi_\f$ is cyclic and $\pi_\f(\ga)\xi_\f$ is a core for it, Lemma \ref{piuno} leads to
\begin{equation}
\label{reverse}
{\rm Ker}(T^{1/2}_{\om/\f})={\rm Ran}\Big(T^{1/2}_{\om/\f}\lceil_{\pi_\f(\ga)\xi_\f}\Big)^\perp=\{0\}\,.
\end{equation}
Then $T^{1/2}_{\om/\f}$ is injective, and so is $T_{\om/\f}$. 

(ii) $\Rightarrow$ (i) By reversing the above computations in \eqref{reverse}, and taking into account that $T_{\om/\f}$ (and also $T^{1/2}_{\om/\f}$) is affiliated with $\pi_\f(\ga)'$, assuming that $T_{\om/\f}$ is injective, we easily deduce that $T_{\om/\f}^{1/2}\xi_\f$ is cyclic.

(ii) $\Rightarrow$ (iv) (ii) says that $T_{\f/\om}^{1/2}$ is injective.  Then
\begin{align*}
\f(a)=&\langle\pi_\f(a)T_{\om/\f}^{-1/2}(T_{\om/\f}^{1/2}\xi_\f),T_{\om/\f}^{-1/2}(T_{\om/\f}^{1/2}\xi_\f)\rangle\\
=&\langle\pi_\om(a)T_{\om/\f}^{-1/2}\xi_\om,T_{\om/\f}^{-1/2}\xi_\om\rangle\,.
\end{align*}
This means that $\f\prec\om$ as well, with $T_{\f/\om}^{1/2}=T_{\om/\f}^{-1/2}$. This also proves the last part.

(iv) $\Rightarrow$ (ii) Suppose (iv) is true and (ii) is false. Then there exists $a\in\ga$ such that
$$
0\neq\eta:=\pi_\f(a)\xi_\f\in{\rm Ker}(T_{\om/\f}^{1/2})\,.
$$
We get $\f(a^*a)=\|\eta\|^2>0$. On the other hand,
$$
\om(a^*a)=\langle T_{\om/\f}^{1/2}\pi_\f(a)\xi_\f,T_{\om/\f}^{1/2}\pi_\f(a)\xi_\f\rangle=0\,,
$$
which contradicts $\f\prec\om$.
\end{proof}

The following results are crucial in the sequel. 
\begin{thm}
\label{twogns}
Let $\ga$ be a $C^*$-algebra, $\a$ a $*$-automorphism and $(\f\circ\a)\prec\f\in\cs(\ga)$. Then there exists a (unique) isometry $U_\a$ acting on $\ch_\f$ satisfying $U_\a U_a^*=P_{(\f\circ\a)/\f}$, such that
$$
U_\a\xi_\f=T^{1/2}_{(\f\circ\a)/\f}\xi_\f\,\,\text{and}\,\, U^*_\a\pi_\f U_\a=\pi_\f\circ\a\,.
$$

In addition, $\f\circ\a\sim\f\iff U_\a$ is unitary.
\end{thm}
\begin{proof}
The uniqueness is trivial, and the last part follows from 
$$
\f\circ\a\sim\f\iff P_{(\f\circ\a)/\f}=I_{\ch_\f}\,.
$$

For the existence, we notice that for the state $\f\circ\a$ (without imposing any further property), $(\pi_\f\circ\a,\ch_\f,\xi_\f)$ is a GNS representation for it. On the other hand, since  $(\f\circ\a)\prec\f$, $(P_{(\f\circ\a)/\f}\pi_\f,P_{(\f\circ\a)/\f}\ch_\f,T^{1/2}_{(\f\circ\a)/\f}\xi_\f)$ is also a GNS representation for $\f\circ\a$. The assertion easily follows as these GNS representations are unitarily equivalent.
\end{proof}
\begin{thm}
\label{multrac}
Fix an automorphism $\a$ and states $\om,\f\in\cs(\ga)$.
\begin{itemize}
\item[(i)] $\f\circ\a\sim\f\iff\f\circ\a^{-1}\sim\f$ and, in such a situation,
\begin{equation}
\label{aamu}
U_{\a^{-1}}=U_\a^*\,,\,\,\text{and}\,\, T_{(\f\circ\a^{-1})/\f}=U^*_\a T_{(\f\circ\a)/\f}^{-1}U_\a\,.
\end{equation}
\item[(ii)] Suppose that $\f\circ\a,\om\prec\f$. Then $\om\circ\a\prec\f\circ\a$ and
\begin{equation*}
T_{\om\circ\a/\f\circ\a}=U_\a T_{\om/\f}U_\a^*
\end{equation*}
\end{itemize}
\end{thm}
\begin{proof}
The proof is a direct consequence of the previous results (in particular, those in Theorem \ref{twogns}, together with the notations therein).

(i) Since 
$$
\f(x)=(\f\circ\a^{-1})(\a(x))\,,\quad x\in\ga\,,
$$
after setting $y:=\a(x)$, we can immediately deduce that
$$
\f\circ\a\sim\f\iff\f\circ\a^{-1}\sim\f
$$
(see Proposition \ref{ortiiq} for a similar analysis). 
Now, \eqref{aamu} is a direct consequence of the previous two theorems.

(ii) Reasoning as in (i), we deduce 
$$
\om\prec\f\iff \om\circ\a\prec\f\circ\a\,.
$$
Since we are also supposing that $\f\circ\a\prec\f$, it is meaningful to consider the Radon-Nikodym derivative $T_{\f\circ\a/\f}$. Now, by taking into account that $U_\a$ is an isometry with range-projection $P_{\f\circ\a/\f}$, and 
$$
\pi_{\f\circ\a}=P_{\f\circ\a/\f}\pi_\f\,,\quad \xi_{\f\circ\a}=T_{\f\circ\a/\f}^{1/2}\xi_\f=U_\a\xi_\f\,,
$$
for $a\in\ga$ we compute
\begin{align*}
&\langle U_\a T^{1/2}_{\om/\f}U_\a^*\pi_{\f\circ\a}(a)\xi_{\f\circ\a},U_\a T^{1/2}_{\om/\f}U_\a^*\xi_{\f\circ\a}\rangle\\
=&\langle T^{1/2}_{\om/\f}U_\a^*P_{\f\circ\a/\f}\pi_\f(a)U_\a\xi_\f,T^{1/2}_{\om/\f}U_\a^*U_\a\xi_\f\rangle\\
=&\langle T^{1/2}_{\om/\f}U_\a^*\pi_\f(a)U_\a\xi_\f,T^{1/2}_{\om/\f}\xi_\f\rangle\\
=&\langle T^{1/2}_{\om/\f}\pi_\f(\a(a))\xi_\f,T^{1/2}_{\om/\f}\xi_\f\rangle\\
=&\om(\a(a))\,.
\end{align*}
Therefore the assertion holds by uniqueness of $T_{\om\circ\a/\f\circ\a}$.
\end{proof}

\section{the KMS condition and the comparison with Pedersen-Takesaki construction}
\label{PTKMS}

The present section is devoted to the bridge between the analysis coming from Modular Theory (e.g. \cite{PT, S})
and applications to Quantum Statistical Mechanics (e.g. \cite{BR}, Vol II). We also discuss such approaches coming from modular theory and the connection with the noncommutative generalisation of dominance and quasi-invariance.
Obviously, both approaches coincide in the abelian situation.

To be more precise, let us consider a state $\f$ on a $C^*$-algebra $\ga$. If a kind of ``Radon-Nikodym'' derivative $W$ is affiliated with 
$\pi_\f(\ga)''$ (or, more precisely, with the centraliser as assumed in \cite{PT})
and the corresponding state 
$\om:=\langle\pi_\f(\,\,\,)W\xi_\f,W\xi_\f\rangle$ is ``strongly absolutely continuous'' w.r.t. the original state $\f$ according to Definition
\ref{dom}, $W$ must be affiliated with 
the centre $\pi_\f(\ga)''\bigcap\pi_\f(\ga)'$ by Theorem \ref{main}. Consequently, if $\f$ is a factor state, the analysis in \cite{PT, S} produces trivial conclusions. In addition, the matter can be reduced to the factor situation by performing the direct integral disintegration of states w.r.t. the centre 
$$
\pi_\f(\ga)''\bigcap\pi_\f(\ga)'\subset \pi_\f(\ga)'\,, 
$$
see Section 3.1 in \cite{S}, and Section IV.6 in \cite{T}.

However, it should be pointed out that, for the applications to Quantum Statistical Mechanics, the decomposition into ``pure thermodynamical phases'' (i.e. decomposition into extremal KMS states) corresponds to the factor decomposition of such KMS states, see e.g. \cite{BR}, Theorem 5.3.30. Therefore, the assumption that the Radon-Nikodym derivative lies in the centre also has important features. 

Indeed, for such a purpose, we have the following
\begin{thm}
\label{pedtake}
Let $\ga$ be a $C^*$-algebra, and consider $\f,\om\in\cs(\ga)$ such that $\f$ has central support. 
The following assertions are equivalent:
\begin{itemize}
\item[(i)] there exists $W\,\eta\,Z(\pi_\f(\ga)'')$ positive, such that $\pi_\f(\ga)\xi_\f$ is a core for $W$ and
$$
\om(a)=\langle\pi_\f(a)W\xi_\f,W\xi_\f\rangle\,,\quad a\in\ga\,;
$$
\item[(ii)] $s(\upsilon_{\xi_\om})\in Z(\pi_\f(\ga)'')$ and 
$\s_t^{\upsilon_{\xi_\om}}=\s_t^{\upsilon_{\xi_\f}}\lceil_{\pi_\f(\ga)''s(\upsilon_{\xi_\om})}$, $t\in\br$;
\item[(iii)] $\upsilon_{\xi_\om}$ satisfies the KMS condition w.r.t. $\s_t^{\upsilon_{\xi_\f}}$;
\item[(iv)] $\om\prec\f$ with $\om$ having central support (and, consequently, 
$P_{\om/\f}=s(\upsilon_{\xi_\om})\in Z(\pi_\f(\ga)'')$), and 
$$
\overline{J_{\upsilon_{\xi_\om}} P_{\om/\f}T^{1/2}_{\om/\f}\lceil_{\D^{1/2}_{\upsilon_{\xi_\f}}\pi_\f(\ga)''\xi_\f}}=T^{1/2}_{\om/\f}J_{\upsilon_{\xi_\f}}\,.
$$
\end{itemize}
\end{thm}
\begin{proof}
Denote $M:=\pi_\f(\ga)''$ with $\upsilon_{\xi_\om}$ and $\upsilon_{\xi_\f}$ the vector state extensions of $\om$ and $\f$ to $M$, respectively. Notice that, taking into account Theorem \ref{main}, (i)-(iii) are equivalent by Corollary 4.11 in \cite{S}, see \cite{PT} for the original version. Therefore, it remains to prove that (iv) is implied by and implies some of the first three equivalent conditions. To shorten the notations, we define $P:=P_{\om/\f}$, and
put $S_{\#}$, $\D_{\#}$, $J_{\#}$, $\s^{\#}$ for the modular objects relative to 
$\upsilon_{\xi_\#}$ for $\#=\om, \f$.

(i)$\Rightarrow$(iv) By Theorem \ref{main}, $\om\prec\f$ with $T^{1/2}_{\om/\f}=W$ by uniqueness. Note also that $\pi_\om(\ga)'$ coincides with the reduced algebra $P\pi_\f(\ga)'P$ acting on 
$\ch_\om=P\ch_\f$. 

Taking into account that $W\,\eta\,Z(M)$, we compute
\begin{align*}
\overline{\pi_\om(\ga)'\xi_\om}=&\overline{PJ_\f\pi_\f(\ga)''J_\f W\xi_\f}=\overline{PWJ_\f\pi_\f(\ga)''\xi_\f}\\
=&\overline{PW\pi_\f(\ga)''\xi_\f}=\overline{P\pi_\f(\ga)''W\xi_\f}\\
=&\overline{P\pi_\f(\ga)''\xi_\om}=\overline{P\pi_\f(\ga)\xi_\om}=\ch_\om\,.
\end{align*}
Hence, $\xi_\om$ is also cyclic for the commutant. Therefore, $\xi_\om$ is standard for $P\pi_\f(\ga)''$, which means that $\om$ has central support (e.g. Theorem \ref{zscsu}).
Pick now $X\in\pi_\f(\ga)''$. Since $\xi_\om$ is separating on 
$\pi_\om(\ga)''=P\pi_\f(\ga)''$, we have 
$$
\|X\xi_{\om}\|^2=0\iff \upsilon_{\xi_\om}(X^*X)=0\,.
$$
By Takesaki's Theorem (e.g. \cite{BR}, Theorem 5.3.10), 
$P\in Z(M)\subset\gc_\f$, where
 $$\gc_\f:=\{x\in M ;\; \sigma^\f_t (x)=x  \quad\forall t\in\mathbb R \}$$
 is the centraliser of the state $v_{\xi_\f}$. Therefore, the modular group $\s^{\f}$ of the state
$\upsilon_{\xi_\f}$ leaves $P$ invariant. For $X\in\pi_\f(\ga)''$,
\begin{align*}
S_\om PWX\xi_\f=X^*\xi_\om=X^*W\xi_\f
=WX^*\xi_\f
=WS_\f X\xi_\f\,,
\end{align*}
which implies $S_{\om} PW=WS_{\f}$ on $\pi_\f(\ga)''\xi_\f$. We also note that $W\,\eta\,Z(\pi_\f(\ga)'')$ and
$\D^{1/2}_\om=\D^{1/2}_\f P=P\D^{1/2}_\f$ (e.g. \cite{S}, Section 4). Therefore, 
on $\pi_\f(\ga)''\xi_\f$ one has
\begin{equation}
\label{laeyu0}
J_\om PW\D^{1/2}_\f=J_\om \D^{1/2}_\om PW=S_\om PW=WS_\f=WJ_\f\D^{1/2}_\f\,,
\end{equation}
which, with $D:=\D^{1/2}_\f\pi_\f(\ga)''\xi_\f$, leads to
\begin{equation}
\label{laeyu}
J_\om PW\xi=WJ_\f\xi\,,\quad \xi\in D\,.
\end{equation}
Now, first note that $WJ_\f$ is a closed operator and, since $J_\f$ is antiunitary and
$$
J_\f D=S_\f\pi_\f(\ga)''\xi_\f=\pi_\f(\ga)''\xi_\f
$$
is a core for $W$, $D$ is a core for $WJ_\f$ leading, by \eqref{laeyu}, to
\begin{equation}
\label{laeyu1}
WJ_\f=\overline{WJ_\f\lceil_D}\,.
\end{equation}
Finally, collecting together \eqref{laeyu1} and \eqref{laeyu0}, we get
$$
WJ_\f=\overline{WJ_\f\lceil_D}=\overline{J_\om PW\lceil_D}\,.
$$

(iv)$\Rightarrow$(iii) Put $W:=T^{1/2}_{\om/\f}$. We already know that $P=s(\upsilon_{\xi_\om})\in Z(M)$, with $M:=\pi_\om(\ga)''$. It is then enough to show that
$\upsilon_{\xi_\om}$ satisfies the KMS condition w.r.t. 
$\s^\f$. Indeed,
for $A,B$ in $M_{\s^\f}$, the weak${}^*$-dense involutive algebra of analytic elements w.r.t. $\s^\f$ (cf. \cite{BR, SZ, S}), we compute
\begin{align*}
\langle AB\xi_\om,\xi_\om\rangle
=&\langle BW\xi_\f,A^*W\xi_\f\rangle
=\langle WB\xi_\f,WA^*\xi_\f\rangle\\
=&\langle WJ_\f\D_\f^{1/2}B^*\xi_\f,WJ_\f\D_\f^{1/2}A\xi_\f\rangle\\
=&\langle J_\om PW\D_\f^{1/2}B^*\xi_\f,J_\om PW\D_\f^{1/2}A\xi_\f\rangle\\
=&\langle W\D_\f^{1/2}A\xi_\f,W\D_\f^{1/2}B^*\xi_\f\rangle\\
=&\langle W\s^\f_{-\iota/2}(A)\xi_\f,W\s^\f_{\iota/2}(B)^*\xi_\f\rangle\\
=&\langle\s^\f_{-\iota/2}(A)W\xi_\f,\s^\f_{\iota/2}(B)^*W\xi_\f\rangle\\
=&\langle\s^\f_{\iota/2}(B)\s^\f_{-\iota/2}(A)\xi_\om,\xi_\om\rangle\,.
\end{align*}
Therefore, $\upsilon_{\xi_\om}$ is a KMS state w.r.t. $\s^\f$ by Proposition 5.3.3 of \cite{BR}.
\end{proof}
\begin{rem}
Notice that, if one of the equivalent conditions of Theorem \ref{pedtake} is satisfied, and in particular (ii), we immediately deduce
\begin{equation*}
\D_{\upsilon_{\xi_\om}}=\D_{\upsilon_{\xi_\f}}\lceil_{P_{\om/\f}\ch_\f}\,,\quad J_{\upsilon_{\xi_\om}}=J_{\upsilon_{\xi_\f}}\lceil_{P_{\om/\f}\ch_\f}\,.
\end{equation*}
\end{rem}

We end the present section by outlining some consequences of the previous results in view of applications to Quantum Field Theory and, mainly, to Quantum Statistical Mechanics concerning the appearance of multiple phases and spontaneous symmetry breaking, see e.g. Vol. 2 of \cite{BR}, and \cite{E}. 

We say that a representation $\pi_1$ of a $C^*$-algebra $\ga$ is {\it normal} w.r.t. another representation $\pi_2$ if there exists a, necessarily normal, $*$-epimorphism $\r:\pi_2(\ga)\to\pi_1(\ga)$ satisfying $\r\circ\pi_2=\pi_1$. The state $\om$ is normal w.r.t. the state $\f$ if $\pi_\om$ is normal w.r.t. $\pi_\f$. The reader is also referred to \cite{T}, Section III.2 for further details.

Notice that, if $\om$ is strongly absolutely continuous w.r.t. $\f$, such a $*$-epimorphism certainly exists and it is given by
$$
\r(X):=XP_{\om/\f}\,,\quad X\in\pi_\f(\ga)''\,.
$$
In the case $\om\sim\f$, we have $\r=\id_{\pi_\f(\ga)''}$. Notice that it is unclear when the reverse property holds true. However, as a direct consequence of the above results, we point out that this is the case for states with central support in the bidual. The forthcoming remarks, whose checking is left to the reader, straightforwardly follow from Theorem \ref{pedtake} by reasoning as in Proposition 3.1 in \cite{BF}. We present the underlying situation taking into account the potential applications to Physics. 

Let $\{\a_t\}_{t\in\br}\subset\aut(\ga)$ be a one-parameter group of $*$-automorphisms acting on $\ga$, and $\f$ a KMS state (i.e. a $\b$-KMS state with $\b=-1$) w.r.t. such a group of automorphisms. In order to avoid the particular situation involving tracial states, however considered, we suppose that $\a_t$ is non-trivial. By Proposition 5.3.12 in \cite{BR},
no a priori continuity property is assumed on the map 
$\br\ni t\mapsto\a_t\in\aut(\ga)$. We also notice that such a $\f$ has central support, and the following equivariant property
$$
\pi_\f\circ\a_t=\s^{\upsilon_{\xi_\f}}_t\circ\pi_\f\,,\quad t\in\br\,,
$$
is fulfilled. The case of a tracial state $\f$ is also of interest. It corresponds to the case $\b=0$ (i.e. infinite temperature in the thermodynamic interpretation) without any condition on the group $\a$, or the case $\a$ trivial without any restriction on the inverse temperature $\b$.
\begin{rem}
\label{vuhlos}
Let $\f\in\cs(\ga)$ be a KMS state w.r.t. $\a$. Then there is a one-to-one correspondence $W\mapsto\f_W$ between:
\begin{itemize}
\item[(a)] the set $\{W\}$ of selfadjoint positive operators affiliated with the centre $\pi_\f(\ga)''\bigcap\pi_\f(\ga)'$ such that $\xi_\f\in\cd_W$ and
$\|W\xi_\f\|=1$;
\item[(b1)] the set $\{\f_W\}$ of KMS states w.r.t. $\a$ which are normal w.r.t. $\f$;
\item[(b2)] the set $\{\f_W\}$ of KMS states w.r.t. $\a$ which are strongly absolutely continuous w.r.t. $\f$.
\end{itemize}
The above assertion still holds by replacing everywhere the sentence ``KMS state(s) w.r.t. $\a$'' with ``tracial state(s)''.
\end{rem}

\section{the chain rule}
\label{secmain2}

From now on, we restrict the matter to states since the aim of the present paper is the investigation of quasi-invariance of states under the actions of groups via $*$-automorphisms.

Quite (not) surprisingly, the relation ``$\prec$'' does not satisfy transitivity. However, we have the following
\begin{thm}
\label{notrans}
Let $\ga$ be a $C^*$-algebra and $\om,\f,\psi\in\cs(\ga)$ with $\om\prec\f$ and $\f\prec\psi$. Define $D:=\pi_\psi(\ga)\xi_\psi$.
Then $\om\prec\psi$ if and only if the operator $A_o:=T^{1/2}_{\om/\f}T^{1/2}_{\f/\psi}\lceil_D$ is closable. 

In such a situation, we have
\begin{equation}
\label{redj}
\left(T^{1/2}_{\om/\f}T^{1/2}_{\f/\psi}\lceil_D\right)^*\overline{\left(T^{1/2}_{\om/\f}T^{1/2}_{\f/\psi}\lceil_D\right)}=T_{\om/\psi}\,.
\end{equation}
\end{thm}
\begin{proof}
Fix a sequence $(x_n)_n\subset\ga$ and put $\xi_n:=\pi_\psi(x_n)\xi_\psi$. Notice that
\begin{align*}
&\|\xi_n\|^2=\psi(x_n^*x_n)\,,\quad \|A_o\xi_n\|^2=\om(x_n^*x_n)\,,\\
&\|A_o(\xi_n-\xi_m)\|^2=\om\big((x_n-x_m)^*(x_n-x_m)\big)\,.
\end{align*}
But this is nothing else than the equivalence between the strong absolute continuity of $\om$ w.r.t. $\psi$, and the closability of $A_o$.

Concerning the last part, the proof ends with computations analogous to those in the proof of Theorem \ref{main}. Since $A_o$ is closable on its core $D$,
with $A:=\overline{A_o}$, it is routine to check that  $A^*A\in\pi_\psi(\ga)'$ and 
$$
W:=(A^*A)^{1/2}=\left(\left(T^{1/2}_{\om/\f}T^{1/2}_{\f/\psi}\lceil_D\right)^*\overline{\left(T^{1/2}_{\om/\f}T^{1/2}_{\f/\psi}\lceil_D\right)}\right)^{1/2}
$$
satisfies 
$$
\om(a)=\langle\pi_\psi(a)W\xi_\psi,W\xi_\psi\rangle\,,\quad a\in\ga\,.
$$

The proof now ends by the uniqueness of the Radon-Nikodym derivative exploited in Theorem \ref{main}.
\end{proof}

Suppose $\om,\f,\psi\in\cs(\ga)$ with $\om\prec\f$, $\f\prec\psi$, and define $P_{\f/\psi}$ and $P_{\om/\f}$ as the selfadjoint projections onto 
$[\pi_\psi(\ga)T^{1/2}_{\f/\psi}]\xi_\psi$ and 
$$
[\pi_\f(\ga)T^{1/2}_{\om/\f}\xi_\f]=[\pi_\psi(\ga)T^{1/2}_{\om/\f}T^{1/2}_{\f/\psi}\xi_\psi]
$$ 
respectively. In the situation of the above theorem, that is 
\begin{equation}
\label{ehkl}
\f\prec\psi\,\,\&\,\,\om\prec\f\Rightarrow \om\prec\psi 
\end{equation}
or, equivalently,
$T^{1/2}_{\om/\f}T^{1/2}_{\f/\psi}\lceil_{\pi_\psi(\ga)\xi_\psi}$ is closable,
we can define the analogous projection onto $[\pi_\psi(\ga)T^{1/2}_{\om/\psi}\xi_\psi]$, obtaining
\begin{equation}
\label{porj}
P_{\om/\psi}=P_{\om/\f}P_{\f/\psi}\,.
\end{equation}
The following comment is in order: \eqref{redj} and \eqref{porj} can be considered as ``chain rules'' for densities and projections, provided that \eqref{ehkl} is satisfied.

\begin{prop}
\label{coidecz}
Suppose that $\psi,\f,\om\in\cs(\ga)$. Then \eqref{ehkl} is satisfied in one of the following situations:
\begin{itemize}
\item[(i)] $\f$ is $\psi$-dominated and $\om$ is strongly $\f$-absolutely continuous;
\item[(ii)] $\psi\sim\f$, $\f\sim\om$, $\xi_\psi\in D(T^{1/2}_{\f/\om})$ and $\eta:=T^{1/2}_{\f/\om}\xi_\psi$ is cyclic for $\pi_\psi(\ga)$;
\item[(iii)] $\ga$ is abelian.
\end{itemize}
\end{prop}
\begin{proof}
(i) follows by Theorem \ref{notrans}, taking into account that $T_{\f/\psi}$ is bounded and thus 
$T^{1/2}_{\om/\f}T^{1/2}_{\f/\psi}\lceil_{\pi_\psi(\ga)\xi_\psi}$ is closable.

(ii) We show that, under such a condition, $A_o^*$ is densely defined. Indeed, on the cyclic subspaces $\pi(\ga)\xi_\psi$, $\pi(\ga)\eta$, we compute
\begin{align*}
\langle T^{1/2}_{\om/\f}T^{1/2}_{\f/\psi}\pi(a)\xi_\psi,\pi(b)\eta\rangle
=&\langle T^{1/2}_{\om/\f}T^{1/2}_{\f/\psi}\pi(a)\xi_\psi,\pi(b)T^{1/2}_{\f/\om}\xi_\psi\rangle\\
=&\langle T^{1/2}_{\om/\f}T^{1/2}_{\f/\psi}\pi(a)\xi_\psi,T^{-1/2}_{\om/\f}\pi(b)\xi_\psi\rangle\\
=&\langle T^{1/2}_{\f/\psi}\pi(a)\xi_\psi,\pi(b)T^{1/2}_{\psi/\f}\xi_\f\rangle\\
=&\langle T^{1/2}_{\f/\psi}\pi(a)\xi_\psi,T^{-1/2}_{\f/\psi}\pi(b)\xi_\f\rangle\\
=&\langle \pi(a)\xi_\psi,\pi(b)\xi_\f\rangle
\end{align*}
Therefore, the dense space $\pi_\psi(\ga)\eta$ is in the domain of $A_o^*$ (and we have
$A_o^*\pi_\psi(x)\eta=\pi_\psi(x)\xi_\f$.)

(iii) If $\ga$ is abelian, the relation ``$\prec$'' is obviously transitive. This can be easily proved as 
each $W\,\eta\,\pi_\f(\ga)'$ is an element of the positive extended part
$\overline{\pi_\f(\ga)'}_+=\overline{\pi_\f(\ga)''}_+$
and, there, the product is well defined.
\end{proof}

\section{quasi-invariance under the action of a group}
\label{inupnkz}

In light of the results in the previous section, from which we deduce that, in general, ``$\prec$'' and even ``$\sim$'' do not satisfy the transitive property, in order to speak and work with ``quasi-invariance under the action of a group'', we must either assume transitivity, or demonstrate that the transitivity is automatically satisfied in this situation. We will show that the latter is indeed the case.

\medskip

Let $\ga$ be a $C^*$-algebra and $G\stackrel{\a_g}{\curvearrowright}\ga$ an action of the group $G$ on $\ga$ by $*$-automorphisms. A state $\f\in\cs(\ga)$ is said to be {\it quasi-invariant under the action} $\a_g$, simply shortened as a ``quasi-invariant state'', if
$$
\f\circ\a_g\sim\f\,;\, g\in G\,.
$$
The set of quasi-invariant states is denoted by $\cs^G(\ga)$. We put $\f_g:=\f\circ\a_g$ and, for $\f\in\cs^G(\ga)$ we denote the set of Radon-Nikodym derivatives, 
$$
T_g:=T_{(\f\circ\a_g)/\f}\equiv T_{\f_g/\f}\\,;\, g\in G\,.
$$

Here, there is the following crucial result:
\begin{prop}
\label{ortiiq}
If $\f\in\cs^G(\ga)$, then 
$\f_g\sim\f_h$ for all $g,h\in G$.
\end{prop}
\begin{proof}
Fix $g,h\in G$. By definition, $\f_{gh^{-1}}\sim\f$. This means that, for each $a\in\ga$,
\begin{align*}
0=&\lim_n\f_h\big(\a_{h^{-1}}(a_n)^*\a_{h^{-1}}(a_n)\big)\\
=&\lim_{m,n}\f_g\big((\a_{h^{-1}}(a_n)-\a_{h^{-1}}(a_m))^*(\a_{h^{-1}}(a_n)-\a_{h^{-1}}(a_m))^*\big)\\
\Rightarrow&\lim_n\f_g\big(\a_{h^{-1}}(a_n)^*\a_{h^{-1}}(a_n)\big)=0
\end{align*}
and, vice-versa,
\begin{align*}
0=&\lim_n\f_g\big(\a_{h^{-1}}(a_n)^*\a_{h^{-1}}(a_n)\big)\\
=&\lim_{m,n}\f_h\big((\a_{h^{-1}}(a_n)-\a_{h^{-1}}(a_m))^*(\a_{h^{-1}}(a_n)-\a_{h^{-1}}(a_m))^*\big)\\
\Rightarrow&\lim_n\f_h\big(\a_{h^{-1}}(a_n)^*\a_{h^{-1}}(a_n)\big)=0\,.
\end{align*}

Now, reading the above computation for $x_n:=\a_{h^{-1}}(a_n)$, we easily obtain the assertion.
\end{proof}
\begin{prop}
\label{concon1}
For the sets $\cs_G(\ga)$ and $\cs^G(\ga)$ of invariant and quasi-invariant states, we have:
\begin{itemize} 
\item[(i)] $\cs_G(\ga)\subset\cs^G(\ga)$;
\item[(ii)] $\cs^G(\ga)$ is convex, but not necessarily ${}^*$-weakly closed.
\end{itemize} 
\end{prop}
\begin{proof}
(i) is immediate, and thus we only have to prove (ii).

The fact that $\cs^G(\ga)$ is not closed follows by the last example in Subsection \ref{otof}. Concerning the convexity, we reason as follows.
Fix $\om,\f\in\cs^G(\ga)$, $\l\in[0,1]$, and put $\psi:=\l\om+(1-\l)\f$. Obviously,
$\l\om$ and $(1-\l)\f$ are $\psi$-dominated, and $(\l\om)\sim(\l\om_g)$, $\big((1-\l)\f\big)\sim\big((1-\l)\f_g\big)$ by quasi-invariance.
Summarising, we get
\begin{equation*}
\left.
\begin{array}{ll}
\,\,\,\,\qquad(\l\om_g)\sim&\!\!\!(\l\om)\\
\big((1-\l)\f_g\big)\sim&\!\!\!\big((1-\l)\f\big)\\
\end{array}
\right\}\prec\psi\,.
\end{equation*}
By (i) in Proposition \ref{coidecz}, we deduce
\begin{equation}
\label{mca}
\left.
\begin{array}{ll}
\,\,\,\,\qquad(\l\om_g)&\\
\big((1-\l)\f_g\big)&\\
\end{array}
\right\}\prec\psi\,.
\end{equation}
By Proposition \ref{concon}, we know that the set of positive functionals strongly absolutely continuous w.r.t. a fixed one is closed under the sum and thus, after summing up the elements on the l.h.s. of \eqref{mca}, we obtain
$$
\psi_g=\l\om_g+(1-\l)\f_g\prec\psi\,.
$$
The proof now ends as follows: as before, since $g\in G$ is arbitrary, we also deduce that $\psi_{g^{-1}}\prec\psi$, which is equivalent to
$\psi\prec\psi_g$. Then $\psi\sim\psi_g$ for all $g\in G$.
\end{proof}

We now pass to the rule satisfied by the involved Radon-Nikodym derivative $T_g\equiv T_{(\f\circ\a_g)/\f}$ for a fixed state $\f\in\cs^G(\ga)$.

Collecting together Proposition \ref{ortiiq}, Theorem \ref{multrac} and Theorem \ref{notrans}, we get the following
\begin{thm}[chain rule]
\label{lenire}
Consider an action $G\stackrel{\a_g}{\curvearrowright}\ga$ of a group $G$ on the $C^*$-algebra $\ga$ by $*$-automorphisms $\a_g$. With $D=\pi_\f(\ga)\xi_\f$, 
\begin{equation}
\label{relautz}
A_o:=U_g T_h^{1/2}U_g^*T_g^{1/2}\lceil_D=U_g T_h^{1/2}T_{g^{-1}}^{-1/2}U_g^*\lceil_D
\end{equation}
is closable, and we get
$$
T_{hg}=A_o^*\overline{A_o}\,.
$$
\end{thm}
\begin{proof}
First note that the equality in \eqref{relautz} directly follows by (i) in Theorem \ref{multrac}.

Now, from Proposition \ref{ortiiq} we know that all states in the orbit $\co_\f$ under the action of $G$ are mutually equivalent. Therefore, for the computation of 
$$
T_{hg}\equiv T_{\f\circ\a_{hg}/\f}=T_{(\f\circ\a_{h})\circ\a_g/\f}\,,
$$
we can apply Theorem \ref{notrans} to the sequence of equivalent states
$$
\f\circ\a_{hg}\equiv(\f\circ\a_{h})\circ\a_g,\,\,\f\circ\a_g,\,\,\f
$$
concluding that the operator
$$
A_o:=T^{1/2}_{\f\circ\a_{hg}/\f\circ\a_g}T^{1/2}_{\f\circ\a_{g}/\f}\lceil_{D}
$$ 
is closable, and
$T_{hg}=(A_o)^*\overline{A_o}$. The assertion now follows by using (ii) in Theorem \ref{multrac}.

The proof ends by uniqueness of the Radon-Nikodym derivative.
\end{proof}
\begin{rem}
\label{lenire0}
Notice that the unitary map 
$$
G\ni g\mapsto U_g\equiv U_{\a_g}\in\cu(\ch_\f)
$$ 
in Theorem \ref{twogns}, merely provides a unitary implementation of $\a_{g^{-1}}$ but it is not a representation of 
the opposite group $G^{\rm op}$. 
\end{rem}
Indeed, for $a\in\ga$ and $g,h\in G$, we get
\begin{align*}
U^*_hU^*_g\pi_\f(a)U_gU_h=U^*_h\pi_\f(\a_g(a))U_h=\pi_\f(\a_h\a_g(a))\\
=\pi_\f(\a_{hg}(a))
=U^*_{hg}\pi_\f(\a_{hg}(a))U_{hg}\,.
\end{align*}
To be a representation of $G^{\rm op}$, one has to check that $U_gU_h\xi_\f=T_{hg}^{1/2}\xi_\f$. But by taking into account (ii) in Theorem \ref{multrac} and the above theorem, we compute
\begin{align*}
U_gU_h\xi_\f=&U_g T_h^{1/2}\xi_\f= U_gT_h^{1/2}U_g^*U_g\xi_\f\\
=&U_gT_h^{1/2}U_g^*T^{1/2}_g\xi_\f=VT^{1/2}_{hg}\xi_\f\,,
\end{align*}
where $V$ is the polar part of the closure of $A_o$ above. Therefore, $U_gU_h\xi_\f$ is, in general, different from $T_{hg}^{1/2}\xi_\f$.

\bigskip

Theorem \ref{lenire} immediately allows us to compute some relevant Radon-Nikodym derivatives. Indeed,
define
$$
B_o:=U_g T_h^{1/2}U_g^*T_g^{1/2}T_h^{-1/2}\lceil_{T_h^{1/2}D}\,.
$$
\begin{rem}
For the following Radon-Nikodym derivatives, we have:
\begin{itemize}
\item[(a)] $T_{\f_{hg}/\f_{g}}=U_gT_hU_g^*$;
\item[(b)] $B_o$ is closable and $T_{\f_{hg}/\f_{h}}=(B_o^*)\overline{B_o}$\,.
\end{itemize}
\end{rem}
\begin{proof}
(a) It is nothing else than (ii) in Theorem \ref{multrac}.

(b) As before, $\xi_{\f_h}=T_h^{1/2}\xi_\f$, and then note that $B_o\big(T_h^{1/2}D\big)=A_o D$ for $A_o$ in Theorem \ref{lenire}. 
Now, first note that, with $B:=\overline{B_o}$, as in Theorem \ref{main}, $B^*B$ is affiliated with $\pi_\f(\ga)'$. Hence 
$(B^*B)^{1/2}=|B|\,\eta\,\pi_\f(\ga)'$. Take now the polar decomposition of $B=V|B|$ and notice that, necessarily, $V$ is unitary by cyclicity. Now, for each $a\in\ga$, we compute
\begin{align*}
\langle\pi_\f(a)|B|\xi_{\f_h},|B|\xi_{\f_h}\rangle=&\langle|B|\pi_\f(a)\xi_{\f_h},|B|\xi_{\f_h}\rangle\\
=&\langle V|B|\pi_\f(a)\xi_{\f_h},V|B|\xi_{\f_h}\rangle\\
=&\langle B\pi_\f(a)\xi_{\f_h},B\xi_{\f_h}\rangle\\
=&\langle B_o\pi_\f(a)\xi_{\f_h},B_o\xi_{\f_h}\rangle\\
=&\langle B_o\pi_\f(a)T_h^{1/2}\xi_{\f},B_oT_h^{1/2}\xi_{\f}\rangle\\
=&\langle B_oT_h^{1/2}\pi_\f(a)\xi_{\f},B_oT_h^{1/2}\xi_{\f}\rangle\\
=&\langle A_o\pi_\f(a)\xi_{\f},A_o\xi_{\f}\rangle\\
=&\langle \pi_\f(\a_{hg}(a))\xi_\f,\xi_\f\rangle\\
=&\f(\a_{hg}(a))\,.
\end{align*}
The proof ends by uniqueness of the Radon-Nikodym derivative.
\end{proof}

In the bounded case, the formula for the chain rule leads to
\begin{equation*}
T_{hg}=T_g^{1/2}U_g T_h U_g^{*}T_g^{1/2}=U_gT_{g^{-1}}^{-1/2}T_h T_{g^{-1}}^{-1/2}U_g^{*}\,.
\end{equation*}

\section{An unbounded implementation}
\label{lenire00}

As usual, we fix a state $\f\in\cs^G(\ga)$. Since $\f$ is, in general, not invariant, we can provide an implementation of the action
of $G$ via unbounded closed operators, in addition to that described in Section \ref{inupnkz}.
By this implementation, the Radon-Nikodym derivatives satisfy a nice cocycle relation. 
If the state under consideration is invariant, both implementations coincide (and the Radon-Nikodym derivatives are obviously trivial).
\begin{thm}
\label{ohfmel}
Let $\f\in\cs^G(\ga)$ and, for each $g\in G$, define on $D=\pi_\f(\ga)\xi_\f$,
$$
V^{(o)}_g:\, \pi_\f(a)\xi_\f\mapsto \pi_\f(\a_g(a))\xi_\f,\quad a\in \ga\,.
$$
The operator $V^{(o)}_g$ is well defined and closable. 

Denoting by $V_g:=\overline{V^{(o)}_g}$, for each $g\in G$ we have:
\begin{itemize}
\item[(i)] $V_g^*V_g=T_g$,
\item[(ii)] $V_g$ is invertible and $V_g^{-1}=V_{g^{-1}}$.
\item[(iii)] the polar decomposition of $V_g$ is $V_g=U^*_g T_g^{1/2}$.
\end{itemize}
As an immediate consequence, we get that the operator 
$T_h^{1/2}V^{(o)}_g=T_h^{1/2}V_g\lceil_D$ is closable, and we get
$$
\left(T_h^{1/2}V_g\lceil_D\right)^*\overline{T_h^{1/2}V_g\lceil_D}=T_{hg}\,.
$$
\end{thm}
\begin{proof}
The first part is proved as follows. Reasoning as in Theorem \ref{main} (where, for $g\in G$ and $\om=\f_g$,
$R_o$ there corresponds to  $V^{(o)}_g$), we immediately conclude that, first $V^{(o)}_g$ is well defined and closable and then
that (i) holds true. 

(ii) directly follows by Theorem \ref{invinv} by taking into account that all the operators
$V_{g}$ have the common core
$$
D=\pi(\ga)\xi_\f=\pi(\a_h(\ga))\xi_\f\,,\quad h\in G\,,
$$
and $\left(V_g\lceil_D\right)^{-1}=V_{g^{-1}}\lceil_D$.

(iii) easily follows by the uniqueness of the polar decomposition of a densely defined closable operator, after computing on elements of the form $\pi_\f(a)\xi_\f$ in the core $D$,
\begin{align*}
V_g\pi_\f(a)\xi_\f=&\pi_\f(\a_g(a))\xi_\f=U^*_g\pi_\f(a)U_g\xi_\f\\
=&U^*_g\pi_\f(a)T^{1/2}_g\xi_\f
=U^*_gT^{1/2}_g\pi_\f(a)\xi_\f\,.
\end{align*}
For the last part, we already know that $A_o$ in Theorem \ref{lenire} is closable and $A_o^*\overline{A_o}=T_{hg}$. Now, by applying (iii), we get
\begin{align*}
T_{hg}=A_o^*\overline{A_o}=&\left(T_h^{1/2}U^*_g T_g^{1/2}\lceil_D\right)^*\overline{T_h^{1/2}U^*_g T_g^{1/2}\lceil_D}\\
=&\left(T_h^{1/2}V_g\lceil_D\right)^*\overline{T_h^{1/2}V_g\lceil_D}\,.
\end{align*}
\end{proof}
In the bounded case, we get
\begin{eqnarray*}
T_{g}\left(V_{g}^{-1}T_{h}V_{g} \right)=T_{hg}=\left(V_{g}^*T_{h}(V_{g}^*)^{-1}\right)T_{g}\,.
\end{eqnarray*}
\begin{rem}
Notice that (i) and (iii) hold true in a more general case: let $\a\in{\rm aut}(\ga)$, $\f\in\cs(\ga)$ with $\f\circ\a\prec\f$. Consider 
the operators $U_\a$, $T_\a$ and $V_\a$, defined as in Theorem \ref{twogns} and Theorem \ref{ohfmel}. 

We then have that
$V^*_\a V_\a=T_\a$ and the polar decomposition of $V_\a$ is $U^*_\a T^{1/2}_\a$.
\end{rem}

\section{Examples}

Here, we collect some examples which contribute to clarifying the emerging situation in the case of dominance, equivalence, and quasi-invariance under the action of a group.

\subsection{Non transitivity}
\label{maov}

The examples for which transitivity does not hold may well exist. Indeed, they arise by taking inspiration from the explicit construction of closed operators whose product is not closable (e.g. \cite{Re}).

With $\mathbb N_+$ denoting the positive integers and $\ell^2\equiv \ell^2(\mathbb N_+)$ , we consider the type I factor $\cb(\ell^2)$ and its standard representation associated to the canonical tracial weight ${\rm Tr}$ (e.g. Section \ref{ptmgv}). In order to construct such an example, we consider positive functionals, disregarding the inessential normalisations.

We put
$$
\xi_{m,n}:=\d_{m,n}/n^3\,,\quad m,n\in\bn_+\,,
$$
and observe that  
$(\xi_{m,n})_{m,n\in\bn_+}=:\xi\in L^2(\cb(\ell^2))$.
We now follow the example in \cite{Re}. First consider $B$, whose matrix-elements are given by
$$
B_{m,n}:=n^2\d_{m,n}\,,\quad m,n\in\bn_+\,.
$$
By Proposition \ref{poseaf}, $\pi_{\rm R}(B)$ is a positive selfadjoint operator, in its own domain, affiliated with $\pi_{\rm L}(\cb(\ell^2))'$ and, in addition,
$$
(\pi_{\rm R}(B)\xi)_{m,n}=\d_{m,n}/n=(\pi_{\rm L}(B)\xi)_{m,n}
$$
(so that $\pi_{\rm R}(B)\xi$ comes from the diagonal embedding of $\ell^2$ in $L^2(\cb(\ell^2)$). Now, for the vector $v\in \ell^2$ with components $v_n=1/n$, consider the rank-one operator $A=\langle\,\,\,,v\rangle v$, which is bounded. As computed in \cite{Re}, 
$\pi_{\rm R}(A)\pi_{\rm R}(B)\equiv\pi_{\rm R}(BA)$
is not closable on
$$
\ell^2_o\subset\cf(\ell^2)\subset\pi_{\rm L}(\cb(\ell^2))\xi\equiv\pi_{\upsilon_\xi}(\cb(\ell^2))\xi_{\upsilon_\xi}
$$
(here, $\ell^2_o$ and $\cf(\ell^2)$ stand for finitely supported sequences and finite rank operators acting on $\ell^2$, respectively, and the second inclusion means the diagonal embedding).
This shows that, with $\eta,\z\in L^2(\cb(\ell^2))$ given by 
$$
\eta:=\pi_R(B)\xi\,,\quad \z:=\pi_R(A)\pi_R(B)\xi\,,
$$
we have, by construction, $\f_{\eta}\prec\f_{\xi}$ (indeed $\f_{\eta}\sim\f_{\xi}$ because $\pi_R(B)$ is injective) and 
$\f_{\z}\prec\f_{\eta}$, but not $\f_{\z}\prec\f_{\xi}$. By replacing $A$ with the injective operator $A+V$ in \cite{Re},
it is also possible to construct a situation for which $\f_{\eta}\sim\f_{\xi}$ and 
$\f_{\z}\sim\f_{\eta}$, but not $\f_{\z}\sim\f_{\xi}$.

\subsection{An example of a quasi-invariant state} 

Consider the $C^*$-algebra $\ga=\mathcal B(\ell^2(\mathbb Z))$. Let $(e_n)_{n\in\mathbb Z}$ be the canonical orthonormal basis of $\ell^2(\mathbb Z)$, together with the canonical matrix-units system $(e_{mn})_{m,n\in\mathbb Z}$. 

On $\ell^2(\mathbb Z)$, define a multiplication operator $\r$, together with its 
diagonal density matrix, denoted also by $\r$, as
\begin{equation}
\label{ohrs}
\rho=\sum_{n\in\mathbb Z}\mu_ne_{nn}\,,\quad
\mu_n=\frac{1}{Z}e^{- n^2},\,\, Z=\sum_{n\in\mathbb Z}e^{-n^2}\,,
\end{equation}
together with the corresponding state 
$$
\f(a)=\tr(\rho a)\,,\quad a\in\ga\,.
$$
Notice that, on the standard representation of $\ga$ (e.g. Section \ref{ptmgv}), the normal state $\f$ is provided by the vector state
$\upsilon_{\r^{1/2}}$.

On $\ell^2(\mathbb Z)$, consider the shift operator $Ve_n=e_{n+1}$, together with its adjoint action on $\ga$. In such a way, we have an action $\bz\stackrel{g_k}{\curvearrowright}\ga$, where
\begin{equation}
\label{ohrs1}
g_k(a):={\rm ad}_{V^k}(a)=V^{k}aV^{-k}\,,\quad k\in\mathbb Z,\,\,a\in\ga\,. 
\end{equation}
The transpose action $g_k^{\rm t}$ on density-matrices, defined as $g_k^{\rm t}(\r):=\r\circ g_k$, is given by ${\rm ad}_{V^{-k}}$ which, for $\r$ in \eqref{ohrs}, leads to
\begin{equation}
\label{dual}
g_k^{\rm t}(\r)=
V^{-k}\rho V^k=\sum_{n\in\mathbb Z}\mu_{n+k}e_{nn}\,.
\end{equation}

Put
\begin{equation}
\label{hk}
h_k=\rho^{-1/2}(V^*)^k\rho V^k\rho^{-1/2} 
\end{equation}
Then, one has
\begin{equation*}
h_ke_n=h_k(n)e_n\,,\quad h_k(n)={\rm exp}(-2 kn- k^2)\,,
\end{equation*}
which is an unbounded operator for $k\neq 0$. 

Let $T^{1/2}_{g_k}$ be the right multiplication by $h_k^{1/2}$:
$T^{1/2}_{g_k}x=xh_k^{1/2}$ in its own domain. It is a positive, selfadjoint, invertible operator affiliated with $\pi_{\rm L}(\ga)'$, see Section \ref{ptmgv}. In particular,
\begin{equation}
\label{norm1}
T^{1/2}_{g_k}\xi_\f=\rho^{1/2}h_k^{1/2}=																\sum_{n\in\mathbb Z}\mu_{n+k}^{1/2}e_{nn}\,.
\end{equation}
Notice that
\begin{eqnarray*}
||T^{1/2}_{g_k}\xi_\f||^2=\sum_{n\in\mathbb Z}\mu_{n+k}=\sum_{n\in\mathbb Z}\mu_{n}=1\,,
\end{eqnarray*}
and thus $\upsilon_{T^{1/2}_{g_k}\xi_\f}$ is a state for each $k\in\bz$. Since $\r$ is invertible, $\f$ is faithful and thus
$\pi_\f(\ga)'=\pi_{\rm L}(\ga)'$ and, consequently, we deduce $T^{1/2}_{g_k}\,\eta\,\pi_\f(\ga)'$ for each $k\in\bz$.
We also note that, by \eqref{hk},
\begin{equation}
\label{rhok}
V^{-k}\rho V^k=\rho^{1/2}h_k\rho^{1/2}=(\rho^{1/2}h_k^{1/2})(h_k^{1/2}\rho^{1/2})\,.
\end{equation}
The properties of the state $\f$ are summarised in the following
\begin{prop}
For the state $\f$ on $\cb(\ell^2(\bz))$ with density $\r$ given in \eqref{ohrs}, the following properties hold true:
\begin{itemize}
\item[(i)] $\f$ is quasi-invariant under the action of $\bz$ given in \eqref{ohrs1};
\item[(ii)] for the unitary implementation $U$ in Theorem \ref{twogns}, we have
\begin{eqnarray*}
U_{g_k}x=V^{-k}xV^k\equiv\pi_{\rm L}(V)^{-k}\pi_{\rm R}(V)^{k}x\,,\quad x\in\mathcal H_\f,\,\,k\in\mathbb Z;
\end{eqnarray*}
\item[(iii)] $U_{g_k}T_{g_l}^{1/2}U_{g_k}^*T_{g_k}^{1/2}$ is essentially selfadjoint on 
$\cd:=\pi_\f(\ga)\xi_\f=\pi_{\rm L}(\ga)\r^{1/2}=\ga\r^{1/2}$, whose closure is $T_{g_{k+l}}^{1/2}$.
\end{itemize}
As a consequence, 
\begin{itemize}
\item[(1)] the chain-rule involving the quasi-invariance of $\f$ (cf. Theorem \ref{lenire}) leads to
$$
T_{g_{k+l}}^{1/2}=U_{g_k}T_{g_l}^{1/2}U_{g_k}^*T_{g_k}^{1/2}
=U_{g_l}T_{g_k}^{1/2}U_{g_l}^*T_{g_l}^{1/2}\,,\quad k,l\in\bz\,;
$$
\item[(2)] $\bz\ni k\mapsto U_{g_k}\in\cu\big(L^2(\cb(\ell^2(\bz)))\big)$ provides a representation of $\bz$.
\end{itemize}
\end{prop}
\begin{proof}
(i) For $a\in\ga$, using repeatedly \eqref{rhok} we get
\begin{align*}
\f\circ g_k(a)=&\langle \pi_\f(g_k(a))\xi_\f,\xi_\f\rangle
              =\langle g_k(a)\rho^{1/2},\rho^{1/2}\rangle\\
							=&\tr(\rho(V^{k}aV^{-k}))
							=\tr((V^{-k}\rho V^{k})a)\\
							=&\tr((a\rho^{1/2}h_k^{1/2})(h_k^{1/2}\rho^{1/2}))
							=\tr((h_k^{1/2}\rho^{1/2})(a\rho^{1/2}h_k^{1/2}))\\
							=&\langle a\rho^{1/2}h_k^{1/2},\rho^{1/2}h_k^{1/2}\rangle
							=\langle \pi_\f(a)T^{1/2}_{g_k}\xi_\f,T^{1/2}_{g_k}\xi_\f\rangle\,.
\end{align*}
This proves that $\f\circ g_k\prec \f$ by Theorem \ref{main}. Since $T_{g_k}$ is injective, by Theorem \ref{invinv}, $\f\circ g_k\sim\f$ and $\f$ is quasi-invariant under the action of $\bz$ given in \eqref{ohrs1}.

(ii) For $x\in\cb(\ell^2(\mathbb Z))\equiv\ga$ and $y\in L^2(\cb(\ell^2(\mathbb Z)))\equiv\ch_\f$, we get
$$
\big(U_{g_k}^*xU_{g_k}\big)y=V^k(x(V^{-k}yV^k))V^{-k}=(V^k xV^{-k})y=g_k(x)y\,.
$$
By taking into account \eqref{norm1} and \eqref{dual}, we have
$$
T^{1/2}_{g_k}\xi_\f=\rho^{1/2}h_k^{1/2}=V^{-k}\rho^{1/2} V^k=U_{g_k}\xi_\f\,,
$$
and thus the conditions in Theorem \ref{twogns} are satisfied.

(iii) Put $A:=U_{g_k}T_{g_l}^{1/2}U_{g_k}^*T_{g_k}^{1/2}$. Taking into account \eqref{norm1},
we compute
\begin{eqnarray*}
A\pi_\f(a)\xi_\f&=&U_{g_k}T^{1/2}_{g_l}U_{g_k}^*T_{g_k}^{1/2}\pi_\f(a)\xi_\f
                  =U_{g_k}V^ka(V^*)^k\rho^{1/2}h_l^{1/2}\\
									&=&(V^*)^{k}V^ka(V^*)^k\rho^{1/2}h_l^{1/2}V^k
									=a(V^*)^k\rho^{1/2}h_l^{1/2}V^k\\
									&=&a\sum_{n\in\mathbb Z}\mu_{n+l}^{1/2}\;(V^*)^ke_{nn}V^k
									=a\sum_{n\in\mathbb Z}\mu_{n+l}^{1/2}\;e_{(n-k)(n-k)}\\
									&=&a\sum_{n\in\mathbb Z}\mu_{n+l+k}^{1/2}\;e_{nn}
									=T^{1/2}_{g_{k+l}}\pi_\f(a)\xi_\f\,.
\end{eqnarray*}
Since $\pi_\f(a)\xi_\f$ is a core for $T^{1/2}_{g_{k+l}}$, it follows that 
$$
\overline{A\lceil_{\pi_\f(a)\xi_\f}}=\overline{T^{1/2}_{g_{k+l}}\lceil_{\pi_\f(a)\xi_\f}}=T^{1/2}_{g_{k+l}}
$$ 
and (iii) holds true.

To end the proof, notice that (1) and (2) are direct consequences of (i)-(iii).
\end{proof}

\subsection{On the orbit of quasi-invariant states.} 
\label{otof}

In the present section, we answer the following question concerning the closure of the orbit $\co_\f$ and its closed convex hull, always in the weak${}^*$-topology. In order to understand what can happen, we reduce our analysis to the (perhaps known to some experts) commutative case.

The first example involves the $C^*$-algebra $\ga=C(\bt)$. For the group we consider all powers of a single rotation by an irrational angle. We describe the framework: we put $\bt\sim[0,1)$ (after identifying $0$ with $1$) equipped with the sum-(mod 1). 

For $\th\in[0,1)$, we consider the map
$$
T_\th(x):=x+\th\,\,\text{(mod 1)}\,,\quad x\in[0,1)
$$ 
such that, on $C(\bt)\equiv\{f\in C[0,1]\,;\,f(0)=f(1)\}$, 
$$
\a(f)\equiv\a_\th(f):=f\circ T_\th\,,\quad f\in C(\bt)\,.
$$
It is clear that the group of integers $\bz$ is acting on $C(\bt)$ via all powers $\a^n$, $n\in\bz$.

For the state $\f$, represented by a probability Radon measure $\m\equiv\m_g$, we take any such measure 
equivalent to the Lebesgue-Haar measure $\l=\di x$: $\di\m_g=g\di x$ with density $g,g^{-1}\in L^1([0,1],\di x)$.
With a slight abuse of notation, we identify the orbit of the state with that of the underlying measure.
\begin{prop}
For the $C^*$-dynamical system $(C(\bt),\a)$, and any measure $\m\sim\l$, $\overline{{\rm hull}(\co_\m)}$ contains an invariant measure equivalent to $\l$. If, in addition, $\th$ is irrational, then $\l\in\overline{{\rm hull}(\co_\m)}$.
\end{prop}
\begin{proof}
The rational case is trivial: it directly follows as $\co_\m$ consists of finitely many points and the transposed action of $\bz$ on it is periodic.

If $\th$ is irrational, it is well known that $(C(\bt),\a)$ is uniquely ergodic with the unique invariant state given by $
\int_{[0,1]}(\,\,\,) \di x$. Therefore (e.g. \cite{Fier}), for each probability Radon measure $\n$ (and then, in particular, for each $\m\sim\l$ as above), 
$$
\lim_n\sum_{k=0}^{n-1}\int_{[0,1]} (f\circ T^k_\th)\di\n=\int_{[0,1]} f(x)\di x\,,\quad f\in C(\bt)\,.
$$
\end{proof}

The next example concerns the opposite situation by considering $C([0,1])$.
We also consider the homeomorphism $g:[0,1]\to[0,1]$ given by $g(x)=\sqrt{x}$. Notice that, first, $g\in C^1((0,1))$. Second, it induces an action of the group $\bz$ made of all integers by considering all powers $g\circ g\circ\cdots\circ g$ and 
$g^{-1}\circ g^{-1}\circ\cdots\circ g^{-1}$. By the Gelfand Theorem, we have an action 
$$
\a_n(f):=f\circ g^n\,,\quad f\in C([0,1]),\,\, n\in\bz\,.
$$
As the reference state, we take $\f:=\int_{[0,1]}(\,\,\,) \di\l$, with
$\di\l$ the Lebesgue measure. We show that the orbit $\co_\f$ of the Lebesgue state under the action inherited from all powers of the homeomorphism $g$ contains cluster points, which are singular w.r.t. the reference state.
\begin{prop}
The Lebesgue state $\f$ is quasi-invariant under the action of $\bz$ induced by (all powers of) the homeomorphism $g$, with Radon-Nikodym derivatives
$$
T_n(x)=2^n x^{2^n-1}\,,\quad n\in\bz\,,
$$
and we have
\begin{equation}
\label{stwa}
\lim_{n\to+\infty}(\f\circ g^n)(f)=f(1),\,\,\lim_{n\to-\infty}(\f\circ g^n)(f)=f(0)\,,\quad f\in C([0,1]).
\end{equation}
\end{prop}
\begin{proof}
The first part follows after performing an elementary change of variable (by taking into account that $g\in C^1((0,1))$).

The second half comes from a standard approximation procedure by invoking the Stone-Weierstrass Theorem. It is then enough to check \eqref{stwa} for monomials $f_k(x):=x^k$, $k\in\bn$. Indeed,
$$
(\f\circ g^n)(f_k)=\int_0^1 T_n(x)x^k\di x=\frac{2^n}{2^n+k}\,,\quad n\in\bz,\,\,k\in\bn\,.
$$
Now, 
$$
\lim_{n\to+\infty}(\f\circ g^n)(f_k)=1=f_k(1),\,\,k\in\bn\,.
$$
For the reverse limit, we put $m:=-n$ obtaining,
\begin{align*}
\lim_{n\to-\infty}(\f\circ g^n)(f_k)=&\lim_{m\to+\infty}(\f\circ g^{-m})(f_k)
=\lim_{m\to+\infty}\frac1{1+2^m k}\\
=&\d_{k,0}=f_k(0),\,\,k\in\bn\,.
\end{align*}
\end{proof}
\begin{rem}
By identifying $\bt$ with $[0,1)$ with the sum-${\rm mod}\,1$ as before, the last example provides a situation for which the weak${}^*$-closure of $\co_\f$ is its one-point compactification, since the points $0$ and $1$ are identified.
\end{rem}

\section*{Acknowledgements}
A. Dhahri is a member of GNAMPA-INdAM and he has been supported by the MUR grant Dipartimento
di Eccellenza 2023-2027 of Dipartimento di Matematica, Politecnico di Milano. F. Fidaleo acknowledges ``Excellence Department Project'' by the Italian Government, CUP E83C23000330006; and ``Tor Vergata University of Rome funding OANGQS'', CUP E83C25000580005. He is also grateful to C. Remling for providing the example in \cite{Re}, and to L. Zsid\'o for a useful discussion about Theorem \ref{zscsu}. The work by H. J. Yoo was supported by the National Research Foundation of Korea(NRF) grant funded by the Korea government(MSIT) (RS-2026-25488431).


\begin{thebibliography}{9999} 

\bibitem{BF} Barreto S. D., Fidaleo F. 
{\it On the structure of KMS states of disordered systems}, 
Commun. Math. Phys. {\bf 250} (2004), 1-21.

\bibitem{BR} Bratteli O., Robinson D. W. {\it Operator algebras and
quantum statistical mechanics I, II}, Springer, Berlin-Heidelberg-New
York, 1987 and 1997.

\bibitem{D} Derezi\'nski J. {\it Unbounded linear operators}. 

\bibitem{E} Emch G. G. {\it Algebraic methods in statistical mechanics and quantum field theory}, Wiley-Interscience, 1972

\bibitem{Fier} Fidaleo F. {\it On strong ergodic properties of quantum dynamical systems}, Infin. Dim. Anal. Quantum Probab. Relat. Top. {\bf 275} (2009), 551-564.

\bibitem{FS} Fidaleo F., Suriano L. {\it Type} III {\it representations and modular spectral triples for the noncommutative torus}, J. Funct. Anal. {\bf 275} (2018), 1484-1531. 

\bibitem{G} Gudder S. P. {\it A Radon-Nikodym Theorem for $*$-algebras}, Pacific J. Math. {\bf 80} (1979), 141-149. 

\bibitem{GK} Guichardet A., Kastler D.. {\it D\'esint\'egration des \'etats quasi-invariants des $C^*$-alg\`ebres}, J. Math. Pures Appl. {\bf 49} (1970), 349-380. 

\bibitem{H} Haagerup U. {\it The standard form of von Neumann algebras}, Math. Scand. {\bf 37} (1975), 271-283. 

\bibitem{H1} Haagerup U. {\it $L_p$-spaces associated with an arbitrary von Neumann algebra}, Colloq. Internat. CNRS, Vol. {\bf 274} (1979), 175-184. 

\bibitem{Hi0} Hiai F. {\it Absolute continuity and disjointness of states in $C^{*}$-dynamical systems}, J. Operator Theory {\bf 11} (1984), 319-331.

\bibitem{Hi} Hiai F. {\it Quantum f-divergence in von Neumann algebras. II}, J. Math. Phys. {\bf 60} (2019), art. n. 012203. 

\bibitem{Kr} Krieger W. {\it On construction of non ${}^*$-isomorphic hyperfinite factors of} {\it Type} III, J. Funct. Anal. {\bf 6} (1970), 97-109. 

\bibitem{N} Naudts J. {\it A generalised entropy function}, Commun. Math. Phys. {\bf 37} (1974), 175-182. 

\bibitem{NSZ} Niculescu C. P., Str\"oh A., Zsid\'o L.
{\it Noncommutative extension of classical and multiple recurrence 
theorems}, J. Operator Theory {\bf 50} (2003), 3-52.

\bibitem{PT} Pedersen G. K., Takesaki M.
{\it The Radon-Nikodym theorem for von Neumann algebras}, Acta Math. {\bf 130} (1973), 53-88.

\bibitem{Pe} Perdrizet F.
{\it \'El\'ements positifs relatifs \`a une alg\`ebre hilbertienne\`a gauche}, Comp. Math. {\bf 23} (1971), 25-47.

\bibitem{Re} https://mathoverflow.net/questions/507863/closure-of-product-of-unbounded-selfadjoint-operators

\bibitem{S}  Str\v{a}til\v{a} \c{S}.  {\it Modular Theory in Operator Algebras}, Editura Academiei and Abacus Press. Tunbridge Wells, (1981).

\bibitem{SZ} Str\v{a}til\v{a}  \c{S}., Zsid\'o L.
{\it Lectures on von Neumann algebras}, Abacus press, Tunbridge
Wells, Kent, (1979).

\bibitem{T} Takesaki M. {\it Theory of operator algebras}, Springer, Berlin-Heidelberg-New
York, 1979.

\end{thebibliography}
\end{document}